\documentclass[twoside]{amsart}

\usepackage{a4}
\usepackage[toc,page]{appendix}
\usepackage{amsmath}
\usepackage{amssymb}
\usepackage{amscd}
\usepackage{amsfonts}
\usepackage{amsthm}
\usepackage{array}
\usepackage[english]{babel}
\usepackage{booktabs}
\usepackage{bbm}
\usepackage{bm}
\usepackage{blkarray}
\usepackage{bmpsize}
\usepackage{cancel}
\usepackage[font=footnotesize]{caption}
\usepackage{color}
\usepackage{enumerate}
\usepackage{enumitem}
\usepackage{extarrows}
\usepackage{graphicx}
\usepackage[driverfallback=dvipdfm, colorlinks=true, citecolor=magenta]{hyperref}
\usepackage[capitalise]{cleveref}
\usepackage{indentfirst}
\usepackage[modulo,right,pagewise]{lineno}
\usepackage[scr=boondox]{mathalfa}
\usepackage{mathdots}
\usepackage{mathrsfs}
\usepackage{mathtools}
\usepackage{mathabx}
\usepackage{multirow}
\usepackage{siunitx}
\usepackage{thmtools}
\usepackage{tikz}
\usepackage{tikz-cd}
\usetikzlibrary{fadings}
\usetikzlibrary{patterns}
\usetikzlibrary{shadows.blur}
\usetikzlibrary{shapes}

\usepackage{verbatim}

\usepackage{tabularx}

\newtheorem{theorem}{Theorem}[section]
\newtheorem{theoremalpha}{Theorem}

\newtheorem{proposition}[theorem]{Proposition}
\newtheorem{corollary}[theorem]{Corollary}
\newtheorem{corollaryalpha}[theoremalpha]{Corollary}
\newtheorem{lemma}[theorem]{Lemma}
\theoremstyle{definition}

\newtheorem{definition}[theorem]{Definition}

\newtheorem{remark}[theorem]{Remark}

\numberwithin{equation}{section}
\numberwithin{theorem}{section}

\DeclareMathOperator{\diam}{diam}

\DeclareMathOperator{\SL}{\mathsf{SL}}

\DeclareMathOperator{\PSL}{\mathsf{PSL}}

\DeclareMathOperator{\Ad}{Ad}

\DeclareMathOperator{\GM}{\mathrm{GM}}
\DeclareMathOperator{\Diam}{\lozenge}

\DeclareMathOperator{\Bc}{\mathcal{B}}
\DeclareMathOperator{\Cc}{\mathcal{C}}
\DeclareMathOperator{\Dc}{\mathcal{D}}

\DeclareMathOperator{\Fc}{\mathcal{F}}
\DeclareMathOperator{\Gc}{\mathcal{G}}

\DeclareMathOperator{\Lc}{\mathcal{L}}

\DeclareMathOperator{\Pc}{\mathcal{P}}
\DeclareMathOperator{\Qc}{\mathcal{Q}}

\DeclareMathOperator{\Nb}{\mathbb{N}}

\DeclareMathOperator{\Rb}{\mathbb{R}}

\newcommand{\de}{\mathrm{d}}
\newcommand{\abs}[1]{\left|#1\right|}

\newcommand{\norm}[1]{\left\|#1\right\|}
\newcommand{\wt}[1]{\widetilde{#1}}
\newcommand{\wh}[1]{\widehat{#1}}

\newcommand{\msf}[1]{\mathsf{#1}}

\newcommand{\vis}[1]{{#1}}

\newcommand{\Hy}{\mathbb{H}}

\newcommand{\N}{\mathbb{N}}

\newcommand{\R}{\mathbb{R}}

\newcommand{\fraka}{\mathfrak{a}}

\newcommand{\frakm}{\mathfrak{m}}
\newcommand{\frakn}{\mathfrak{n}}

\newcommand{\sfc}{\mathsf{c}}
\newcommand{\sfL}{\mathsf{L}}
\newcommand{\sft}{\mathsf{t}}
\newcommand{\vv}{\mathbf{v}}
\newcommand{\vw}{\mathbf{w}}
\newcommand{\ve}{\mathbf{e}}

\newcommand{\al}{\alpha}

\newcommand{\gam}{\gamma}

\newcommand{\del}{\delta}
\newcommand{\ep}{\epsilon}
\renewcommand{\epsilon}{\varepsilon}

\newcommand{\thet}{\theta}

\newcommand{\Lam}{\Lambda}

\newcommand{\lam}{\lambda}

\newcommand{\sig}{\sigma}

\newcommand{\p}{\psi}

\newcommand{\Om}{\Omega}

\newcommand{\om}{\omega}

\newcommand{\bs}{\backslash}

\newcommand{\ra}{\rightarrow}

\newcommand{\ov}[1]{\overline{#1}}

\begin{document}

\title[Manhattan manifolds and exact dimensionality]{Regularity of Manhattan manifolds and exact dimensionality for relatively Anosov groups}

\subjclass[2020]{22E40, 20F67, 20F69, 28A78, 37D40}

\keywords{Relatively Anosov groups, Patterson--Sullivan measures, exact dimensionality, Manhattan manifolds, growth indicators}

\author{Eduardo Reyes}
\address{Facultad de Matem\'aticas, Pontificia Universidad Cat\'olica de Chile (PUC)}
\email{eduardoreyes@uc.cl}
\thanks{}

\author{Tianqi Wang}
\address{Department of Mathematics, Yale University, New Haven, CT, USA}
\email{tq.wang@yale.edu}
\thanks{}


\begin{abstract}
    We establish a dynamical framework based on first return maps of Bowen--Margulis--Sullivan flows for the study of Patterson--Sullivan measures and growth indicators of relatively Anosov groups. This framework allows us to deduce several regularity results. First, we prove that these measures are exact dimensional with respect to visual metrics induced by Gromov models in the Groves--Manning quasi-isometry class. Under the additional assumption that the group is relatively Morse, we show that the associated scalar Cartan metric is Gromov hyperbolic. Consequently, the corresponding boundary premetric is a visual metric to which the exact dimensionality theorem applies. Second, we prove that the associated Manhattan manifolds are $C^1$-regular, from which we deduce that the growth indicator is $C^1$-regular and strictly concave on the interior of the limit cone. These results extends the corresponding results for Anosov representations. 

\end{abstract}

\maketitle

\section{Introduction}

Patterson--Sullivan measures were first introduced by Patterson \cite{patterson1976limit} for Fuchsian groups and later by Sullivan \cite{sullivan1979density} for discrete subgroups of isometries of $\Hy^n$, providing a link between orbit growth, the large-scale geometry of the group, and the dynamics of the corresponding geodesic flows. Among the many generalizations of this theory, Albuquerque \cite{albuquerque1999patterson} and Quint \cite{quint2002mesures} extended this notion to discrete subgroups of semisimple Lie groups of higher rank. In this setting, this paper studies Patterson--Sullivan measures for \emph{relatively Anosov groups}, a prominent class of groups introduced by Kapovich--Leeb \cite{kapovich2018relativizing} and Zhu--Zimmer \cite{zhu2022relatively}, which can be viewed as a higher rank generalization of geometrically finite groups in rank $1$. 

We work in the following setting, for which we refer to \Cref{section: preliminaries} for more details. Let $\msf G$ be a connected semisimple real algebraic group. We fix a Cartan decomposition $\msf G=\msf K\exp(\fraka^+)\msf K$ and denote by $\mu:\msf G\to \fraka^+$ the Cartan projection. We fix a non-empty set $\theta$ of simple roots and denote by $\msf P_\theta< \msf G$ the associated standard parabolic subgroup, by $\Fc_\theta=\msf G/ \msf P_\theta$ the $\theta$-flag manifold, and by $\mu_\theta: \msf G\to \fraka_\theta^+$ the $\theta$-Cartan projection.

Let $\Gamma < \msf G$ be a subgroup that is $\theta$-Anosov relative to a collection $\Pc$ of subgroups, so that the pair $(\Gamma,\Pc)$ is relatively hyperbolic. We denote by $\Lc_\theta(\Gamma)\subset\fraka_\theta^+$ the $\theta$-limit cone of $\Gamma$ and by $\Lambda_\theta\subset \Fc_\theta$ the $\theta$-limit set of $\Gamma$. For $\psi\in \fraka_\theta^\ast$, a $(\Gamma,\psi)$-\emph{Patterson--Sullivan measure} is a Borel probability measure on $\Lam_\theta$ such that for all $\xi\in \Lam_\theta$ and $\gamma\in \Gamma$ we have
\[\frac{\de \gam_\ast \nu}{\de \nu}(\xi)=e^{\psi(\beta_\xi^\theta(e,\gamma))}~.\]
Here $\beta^\theta$ denotes the $\fraka_\theta$-valued Busemann function, see \Cref{section: Lie groups}. Suppose that $\psi$ is positive on $\Lc_\theta(\Gamma)\setminus\{0\}$ and its critical exponent  \[\delta_\psi(\Gamma)=\limsup_{T\to\infty} \frac{1}{T} \log \#\{\gamma\in\Gamma: \psi(\mu_\theta(\gamma))\leqslant T\}~\]
is $1$. In this case, Canary--Zhang--Zimmer \cite{canary2025patterson} proved existence and uniqueness of the $(\Gamma,\psi)$-Patterson--Sullivan probability measure, which we denote by $\nu_\psi$. The purpose of this paper is to study several properties of Patterson--Sullivan measures of relatively Anosov groups.

\subsection{Exact dimensionality}
Let $\Gamma<\msf G$ be a $\theta$-Anosov group relative to $\Pc$. Then the action of $\Gamma$ on the symmetric space $(\msf G/\msf K,\msf{d}_{\msf G/\msf K})$ is \emph{orbit-quasi-isometric} to the action of $\Gamma$ on a Groves--Manning cusped space $(X_{\GM},\msf{d}_{\GM})$ \cite{groves2008dehn}. That is, for any $o\in X_{\GM}$ and $o'\in \msf G/\msf K$ there are constants $C\geqslant 1$ and $c\geqslant 0$ such that for any $\gamma\in \Gamma$ we have \[C^{-1}\msf{d}_{\msf G/\msf K}(o', \gamma o')-c \leqslant \msf{d}_{\GM}(o,\gamma o) \leqslant C \msf{d}_{\msf G/\msf K}(o', \gamma o') + c~,\] see \cite[Theorem~1.7, Proposition~1.13]{zhu2022relatively}.
Indeed, our results apply to any Gromov model $(Y,\msf{d}_Y)$ for $(\Gamma,\Pc)$ that is $\Gamma$-equivariantly quasi-isometric to $(X_{\GM},\msf{d}_{\GM})$. 
We denote by $\vis{d}_Y$ a visual quasi-metric on $\partial Y$ associated to $\msf{d}_Y$. That is, \[\vis{d}_Y(\xi,\eta)= \begin{cases}e^{-(\xi,\eta)_{y_0}} & \text{if } \xi\neq \eta~,\\ 0 & \text{if } \xi=\eta~,\end{cases}\] where $(\cdot,\cdot)_{y_0}:\partial Y \times \partial Y\ra \R$ is an extension of the Gromov product based at $y_0\in Y$. For this quasi-metric we denote $B(\xi,r)=\{\eta\in \partial Y: \vis{d}_Y(\xi,\eta)<r\}$. We also identify the Bowditch boundary $\partial Y = \partial (\Gamma,\Pc)$ of $(\Gamma,\Pc)$ and the $\theta$-limit set $\Lambda_\theta$ of $\Gamma$. This allows us to regard any $(\Gamma,\psi)$-Patterson--Sullivan measure for $\psi\in \fraka_\theta^\ast$ as a measure on $\partial Y$.

Our first main result establishes the exact dimensionality of Patterson--Sullivan measures with respect to visual quasi-metrics arising from the above class of Gromov models.

\begin{theoremalpha}\label{theoremalpha: exact dimensionality}
    Let $\Gamma < \msf G$ be a $\theta$-Anosov group relative to $\Pc$, and let $(Y,\msf{d}_Y)$ be a Gromov model for $(\Gamma,\Pc)$ that is $\Gamma$-equivariantly quasi-isometric to a Groves--Manning cusped space. If $\psi\in \fraka_\theta^\ast$ is a $(\Gamma,\psi)$-proper linear form that has critical exponent $1$, then the $(\Gamma,\psi)$-Patterson--Sullivan measure $\nu_\psi$ is exact dimensional with respect to $\vis{d}_Y$. That is, there exists $h>0$ such that for $\nu_\psi$-almost every $\xi\in \partial Y$ we have \[h=\lim_{r\to0}\frac{\log\nu_\psi(B(\xi,r))}{\log r}~.\] In particular, $h$ equals the Hausdorff dimension of the measure $\nu_\psi$ with respect to $\vis{d}_Y$.   
\end{theoremalpha}

In some situations, exact dimensionality follows from Ahlfors regularity of the measure, which implies that the relevant measure is equivalent to the Hausdorff measure of the space. This is the case of Patterson--Sullivan measures for symmetric linear forms associated to Anosov groups, as established by Dey--Kim--Oh \cite[Theorem~1.1]{dey2024ahlfors}. When $\psi$ is not symmetric, the Patterson--Sullivan measure is still $s$-Ahlfors regular for some $s\in (0,1)$, but is no longer comparable to a Hausdorff measure \cite[Theorem~1.3]{dey2024ahlfors}. In the presence of cusps, a comparison between Patterson--Sullivan measures and Hausdorff measures depends heavily on the growth behavior of the cusp subgroups. In the rank $1$ case, this is Sullivan's criterion \cite{sullivan.entropy}, which was recently extended to relative Morse groups by Kim--Oh \cite{kim-oh.globalshadow}. Our \Cref{theoremalpha: exact dimensionality} extends the exact dimensionality of all these previously known cases, and does not rely on Ahlfors regularity. We note that exact dimensionality has also been established for harmonic measures associated to random walks on groups. For actions on hyperbolic spaces, this is due to Tanaka \cite{tanaka.dimension}, whereas for random walks on $\SL(n,\Rb)$ this was proven by Ledrappier--Lessa \cite{ledrappier-lessa}.

\subsection{Regularity of Manhattan manifolds and strict concavity of growth indicators}

The second goal of the paper is to prove regularity of the Manhattan manifolds for relatively Anosov groups. These manifolds generalize the Manhattan curves introduced by Burger \cite{burger} for pairs of convex-cocompact representations in rank $1$ Lie groups. There are recent works on Manhattan curves for some relative Anosov groups, such as \cite{kao.1,kao.2} for pairs of cusped Fuchsian representations, \cite{BCK} for pairs of cusped quasi-Fuchsian representations, \cite{kao.3} that compares quasi-Fuchsian representations with negatively curved metrics on surfaces, and \cite{BCKM} for pairs of cusped Hitchin representations.
Manhattan manifolds were introduced by Cantrell, the first named author and Sert in \cite{CRS} for tuples of left-invariant metrics on hyperbolic groups, generalizing Manhattan curves in this setting by Cantrell--Tanaka \cite{cantrell-tanaka.manhattan}.

Let $\Gamma<\msf G$ be relatively $\theta$-Anosov as above. A linear form $\psi\in \fraka_\theta^\ast$ is $(\Gamma,\theta)$-\emph{proper} if for any $C>0$ there are at most finitely many $\gamma\in \Gamma$ satisfying $\psi(\mu_\theta(\gamma))\leqslant C$. Given $(\Gamma,\theta)$-proper linear forms $\psi_0,\psi_1,\dots,\p_m \in \fraka_\theta^\ast$, their \emph{Manhattan manifold parametrization} is the function $\Theta:\R^m \ra \R$ such that for $\vv=(v_1,\dots,v_m)\in \R^m$, $\Theta(\vv)$ is the critical exponent of the series \[s \mapsto \sum_{\gamma\in \Gamma}{e^{-\sum_{i=1}^m{v_i\p_i(\mu_\theta(\gamma))}-s\p_{0}(\mu_\theta(\gamma))}}~.\]

Our next main result establishes differentiability of this parameterization.

\begin{theoremalpha}\label{theoremalpha: regularity of Manhattan manifold}
    Suppose $\Gamma < \msf G$ is a relatively $\theta$-Anosov group and $\psi_0,\psi_1,\dots,\p_m \in \fraka_\theta^\ast$ are $(\Gamma,\theta)$-proper linear forms. Then the Manhattan manifold parametrization $\Theta:\Rb^m \to \Rb$ of the tuple $(\p_0,\dots,\p_m)$ is $C^1$-regular.
\end{theoremalpha}

For $m=1$, this provides the first proof of $C^1$-regularity of Manhattan curves for pairs of relative Anosov groups without relying on thermodynamic formalism, which is not yet available in this setting. On the other hand, thermodynamic formalism has been used to prove real analyticity of Manhattan curves for pairs of cusped Fuchsian representations \cite{kao.1,kao.2}, cusped Hitchin representations \cite{BCKM}, cusped quasi-Fuchsian representations \cite{BCK}, and Anosov representations \cite{cantrell-tanaka.invariant} (see also \cite{sambarino2014hyperconvex}). 

\begin{remark}
    While working on the final draft of this work, we learned that Dongryul Kim, Hee Oh and Andrew Zimmer have independently proved \Cref{theoremalpha: regularity of Manhattan manifold} \cite{kim-zimmer.C1}.
\end{remark}

A consequence of the $C^1$-regularity of Manhattan manifolds is the strict concavity and $C^1$-regularity of the growth indicator. The $\theta$-\emph{growth indicator} $\p_\Gamma^{\theta}:\fraka_\theta \ra \R \cup \{-\infty\}$, introduced by Quint \cite{quint2002divergence}, is a higher rank analogue of the critical exponent of $\Gamma$. It records the exponential growth rate of $\mu_\theta(\Gamma)$ along each direction of $\fraka_\theta$ (see \Cref{section: Patterson--Sullivan measures}). Let\[\Dc_\Gamma^\theta = \{\psi\in\fraka_\theta^\ast: \psi\geqslant \psi_\Gamma^\theta\}\] be the \emph{dual growth domain}. Its boundary consists of $(\Gamma,\theta)$-proper linear forms with critical exponent $1$. By choosing the forms $\p_0,\dots,\p_m$ as a basis of $\fraka_\theta^\ast$, the function $\Theta:\R^m \ra \R$ serves as a parameterization of the boundary of $\Dc_\Gamma^\theta$ (for Anosov groups this was noted in \cite{CRS}). Therefore, \Cref{theoremalpha: regularity of Manhattan manifold} implies that this boundary is a $C^1$-hypersurface. Together with the duality between supporting linear forms and the growth indicator developed by Quint \cite{quint2003indicateur} and Sambarino \cite{sambarino2014hyperconvex}, we deduce the following.

\begin{corollaryalpha}\label{corollaryalpha: strict concavity of growth indicator}
    Let $\Gamma <\msf G$ be a Zariski dense relatively $\theta$-Anosov group. Then the $\theta$-growth indicator $\psi^\thet_\Gamma: \Lc_\theta(\Gamma) \to [0,\infty)$ is strictly concave and $C^1$ on the interior of $\Lc_\theta(\Gamma)$.
\end{corollaryalpha}

Note that $\psi^\thet_\Gamma$ is homogeneous along any direction in $\Lc_\theta(\Gamma)$ by definition. Here the strict concavity of $\psi^\thet_\Gamma: \Lc_\theta(\Gamma) \to [0,\infty)$ means that for any $\vv_1,\vv_2 \in \Lc_\theta(\Gamma)\setminus \{0\}$ that are non-parallel, the function $[0,1]\ni t\mapsto  \psi^\thet_\Gamma ((1-t) \vv_1 + t \vv_2)$ is strictly concave. This projective strict concavity appears to be new, while the $C^1$-regularity of $\psi^\thet_\Gamma$ is also an immediately consequence of \cite[Corollary~1.9, Corollary~1.10]{canary2025patterson}.

This result also recovers the strict  concavity and differentiability of growth indicators for Zariski dense Anosov representations, first proven by Quint for Schottky groups \cite{quint2003indicateur}, then by Sambarino for hyperconvex representations \cite{sambarino2014hyperconvex}, and finally by Kim--Oh--Wang for arbitrary Anosov groups \cite[Theorem~13.2]{kim2025properly}.

\subsection{A hyperbolic metric for relatively Morse groups}\label{sectionintroduction: relatively Morse groups}

We assume now that $\theta$ is symmetric, i.e., $\theta = i(\theta)$, where $i$ is the opposite involution on the root system (see \Cref{section: Lie groups}). Suppose $\Gamma< \msf G$ is relatively $\theta$-Anosov and let $\varphi\in \fraka_\theta^\ast$ be a symmetric, $(\Gamma,\theta)$-proper linear form. Dey--Kim--Oh \cite{dey2024ahlfors} considered the premetric defined on the limit set $\Lambda_\theta$ of $\Gamma$ by \[\vis{d}_\varphi(\xi,\eta)= \begin{cases}e^{-\varphi(G^\theta(\xi,\eta))} & \text{if } \xi\neq \eta~,\\ 0 & \text{if } \xi=\eta~.\end{cases}\] Here $G^\theta$ denotes the $\fraka_\theta$-valued Gromov product as in \Cref{section: Lie groups}. When $\Gamma$ is $\theta$-Anosov, they proved the Ahlfors regularity of $\nu_\varphi$ with respect to $\vis{d}_\varphi$ \cite[Theorem~1.1]{dey2024ahlfors}.

We further study the premetric $d_\varphi$ in the setting of relatively Morse groups. Recall from the work of Kapovich--Leeb \cite{kapovich2018relativizing} that $\Gamma <\msf G$ is a \emph{relatively $\theta$-Morse group} if it admits a relatively hyperbolic group structure $(\Gamma,\Pc)$ and there exists a Gromov model $Y$ of $(\Gamma,\Pc)$ and a $\Gamma$-equivariant $\theta$-Morse quasi-isometric embedding $F:Y\to \msf G/\msf K$. Here $\theta$-Morse means that the image of every geodesic segment in $Y$ lies in a uniform neighborhood of the corresponding $\theta$-diamond in $\msf G/\msf K$ (see \Cref{section: relatively Morse groups} for details). Relatively Morse groups are relatively asymptotically embedded \cite[Theorem~8.3]{kapovich2018relativizing} and hence relatively Anosov \cite[Proposition~4.4]{zhu2022relatively}.

Relatively Morse groups form a natural class of relatively Anosov groups for which the orbit geometry in the symmetric space is better controlled, and contain all classical Anosov groups. Although it remains unknown whether the relatively Morse condition always holds for relatively Anosov groups, groups satisfying these conditions are plentiful. Zhu--Zimmer \cite{zhu2024relatively} showed via an even stronger condition called ``uniformly Anosov'' that 
\begin{itemize}
    \item If $\Gamma< \msf G$ is relatively $\theta$-Anosov and the relatively hyperbolic group structure of $\Gamma$ is coming from a Fuchsian representation $\Gamma\to \PSL(2,\Rb)$, then $\Gamma$ is $\theta$-Morse \cite[Corollary~1.14]{zhu2024relatively}. This is the case for the cusped Hitchin representations introduced in \cite{canary2022cusped}.
    \item If $\msf G$ is of rank $1$ and $\tau: \msf G\to \SL(n,\Rb)$ is a representation such that the image of $\tau$ contains a $\msf P_k$-proximal element, then for any geometrically finite subgroup $\Gamma< \msf G$, one has that $\tau(\Gamma)$ is $\msf P_k$-Morse \cite[Proposition~1.10]{zhu2024relatively}.
\end{itemize}

We now assume that $\Gamma < \msf G$ is relatively $\theta$-Morse and $F:Y\to \msf G/\msf K$ is a $\Gamma$-equivariant $\theta$-Morse quasi-isometric embedding from a Gromov model $Y$ to the symmetric space. For any $g \msf K, h \msf K \in \msf G / \msf K$, we define \[D_\varphi(g \msf K,h \msf K) = \varphi(\mu_\theta(g^{-1}h)),\] which is independent of the choice of $g$ and $h$, and for any $x,y\in Y$ we set \[\msf{d}_\varphi(x,y) = D_\varphi(F(x),F(y))~.\]

Our next result asserts that $\msf d_\varphi$ is a hyperbolic metric, up to a uniformly bounded additive perturbation.

\begin{theoremalpha}\label{theoremalpha: relativemorse implies hyperbolic}
    Suppose $\Gamma<\msf G$ is a relatively $\theta$-Morse group and $\varphi\in \fraka_\theta^\ast$ is a symmetric, $(\Gamma,\theta)$-proper linear form. Then there exists a constant $C>0$ such that 
    \begin{equation*}
    \ov{\msf{d}}_\varphi(x,y) =
        \begin{cases}
            \ \msf{d}_{\varphi}(x, y)+ C & \text{ if }x\ne y\\
            \  0 & \text{ if }x=y
        \end{cases}
    \end{equation*} is a $\Gamma$-invariant, Gromov hyperbolic metric on $Y$. Moreover, $(Y,\ov{\msf{d}}_\varphi)$ is a Gromov model for $(\Gamma,\Pc)$ that is $\Gamma$-equivariantly quasi-isometric to a Groves--Manning cusped space (see \Cref{remark: gromov models}).
\end{theoremalpha}

We call $\msf d_\varphi$ and also its bounded perturbation $\ov{\msf{d}}_\varphi$ \emph{scalar Cartan metrics} on $Y$. When $\Gamma$ is $\theta$-Anosov, this result was essentially proven in \cite[Corollary~4.8]{dey-kapovich} (see also \cite{jyothis-martinezgranado}).

The proof of \Cref{theoremalpha: relativemorse implies hyperbolic} is based on the Morse geometry of the orbit map developed in a series of papers of Kapovich--Leeb \cite{kapovich2017discrete, kapovich2018finsler, kapovich2018relativizing} and Kapovich--Leeb--Porti \cite{kapovich2014morse, kapovich2016some, kapovich2017anosov, kapovich2017discrete, kapovich2018morse}. For a symmetric, $(\Gamma,\theta)$-proper linear form $\varphi$, the relative $\theta$-Morse condition implies that the scalar Cartan metric is coarsely additive along points lying near geodesics in $Y$. Some of these estimates also appear in the recent work of Kim--Oh \cite[Proposition~4.9, Lemma~4.10]{kim-oh.globalshadow}.
Combining this coarse additivity with local boundedness of $\msf d_\varphi$ and the hyperbolicity of $Y$, the hyperbolicity of the scalar Cartan metric follows from a standard criterion.

From \Cref{theoremalpha: relativemorse implies hyperbolic} we deduce that, when restricted to the Gromov model $Y$ (seen as a subset of $\msf G /\msf K$ via $F$), the Busemann function $\varphi \circ B^\theta$ and Gromov product $\varphi \circ G^\theta$ defined in the Lie theoretic sense coincide with the Busemann function and Gromov product of $(Y,\ov{\msf{d}}_\varphi)$ in the classical sense of Gromov hyperbolicity (see \Cref{subsection:compatibility Busemann}). In particular, $\vis{d}_\varphi$ is a visual quasi-metric associated to $\ov{\msf{d}}_\varphi$. This fact together with \Cref{theoremalpha: exact dimensionality} immediately imply the following.

\begin{corollaryalpha}\label{corollaryalpha: premetric exact dimensionality}
    Suppose $\Gamma<\msf G$ is a relatively $\theta$-Morse group and $\varphi,\psi\in \fraka_\theta^\ast$ are $(\Gamma,\theta)$-proper, positive linear forms. If $\varphi$ is symmetric and $\delta_\psi(\Gamma)=1$, then the Patterson--Sullivan measure $\nu_{\psi}$ is exact dimensional with respect to $\vis{d}_\varphi$.
\end{corollaryalpha}

In particular, the Patterson--Sullivan measure $\nu_\varphi$ can be seen as a quasi-conformal measure for $(Y,\ov{\msf d}_\varphi)$. Since $(Y,\ov{\msf d}_\varphi)$ is also a Gromov model for $(\Gamma,\Pc)$, some recent results for Patterson--Sullivan theory of relatively Morse groups are indeed instances of the analogous results for geometrically finite actions on Gromov hyperbolic spaces. For example, the Ahlfors regularity of Patterson--Sullivan measures for Anosov representations by Dey--Kim--Oh \cite{dey2024ahlfors} follows from the corresponding assertion for quasi-conformal measures for geometric actions of hyperbolic groups due to Coornaert \cite[Proposition~7.4]{coornaert}.

\medskip
\subsection{Comments on the proofs of Theorems \ref{theoremalpha: exact dimensionality} and \ref{theoremalpha: regularity of Manhattan manifold}}

We briefly explain the methods in the proof of exact dimensionality of the Patterson--Sullivan measures and regularity of Manhattan manifolds for relatively Anosov groups. Our approach is dynamical, inspired by the proof of Cantrell--Tanaka \cite{cantrell-tanaka.manhattan} of regularity of Manhattan curves for pairs of metrics on hyperbolic groups.

We start with the flow space $(\wt \Omega,\phi)$, where $\wt \Omega=\Lam_{\theta}^{(2)}\times \R$ and $\Lambda_\theta^{(2)}$ is the set of transverse pairs in $\Lam_\theta \times \Lam_{i(\theta)}$, and the flow $\phi$ is induced by translations in the $\R$-coordinate. This is the analog of the geodesic flow on a rank $1$ symmetric space. Each $(\Gamma,\theta)$-proper linear form $\psi\in \fraka_\theta^\ast$ induces a properly discontinuous action of $\Gamma$ on $\wt \Omega$ that commutes with $\phi$, and hence the quotient $\Om_\psi=\Gamma \bs \wt \Om$ is also equipped with a flow $\phi$. This construction is due to Kim--Oh--Wang \cite{kim2025properly}. The system $(\Om_\psi,\phi)$ is also equipped with a natural $\phi$-invariant measure, namely the \emph{Bowen--Margulis--Sullivan measure} $\frakm_\psi$, which is built from the Patterson--Sullivan measures $\nu_\psi$ and $\nu_{\psi \circ i}$. A key consequence of the relative $\theta$-Anosov property is that the measure $\frakm_\psi$ is finite and mixing, as proven by Kim--Oh \cite{kim2025relatively} and independently in the GPS framework by Blayac--Canary--Zhu--Zimmer \cite{blayac2024counting}. This allows us to describe the dynamics of $(\Om_\psi,\phi,\frakm_\psi)$ in terms of a first return map to a compact subset $K\subset \Om_\psi$. By choosing a precompact Borel lift of $K$ in the cover $\wt \Om$, we also obtain a measurable cocycle $\msf{c}:\N \times K\ra \Gamma$. The linear form $\psi$ and the action of $\Gamma$ on $(Y,\msf d_Y)$ also induce almost subadditive functions on $\Gamma$, which, when composed with $\msf c$, yield two almost subadditive sequences of measurable real functions on $K$. The exact dimensionality of $\nu_\psi$ then follows by applying the Kingman's subadditive ergodic theorem to these sequences (with respect to $\frakm_\psi$), combined with the shadow lemma for relatively Anosov groups.

The proof of \Cref{theoremalpha: regularity of Manhattan manifold} uses these same techniques, with the extra difficulty that now we have a family of flow spaces $\vv\mapsto (\Om_\vv,\phi,\frakm_\vv)$ constructed from the assignment $\R^m \ni \vv \mapsto \p_\vv\in \fraka_\theta^\ast$ given by the Manhattan parameterization $\Theta:\R^m \ra \R$ (see \Cref{eq:p_vv}). We therefore fix the reference flow space $\Omega_{0}$ associated to $\psi_0$ and place all Bowen--Margulis currents on this single space. Locally, the flow and the first return map are fixed, while only the new Bowen--Margulis--Sullivan measure $\frakm_\vv^0$ on $\Om_{\psi_0}$ varies with $\vv$. We then verify the finiteness and ergodicity of these measures, the  weak* continuity of the Patterson--Sullivan measure $\nu_\vv$, and the vague continuity of the Bowen--Margulis--Sullivan measures $\frakm_\vv^0$ as $\vv$ varies. Combined with the coarse reparameterization theorem by Kim--Oh (\Cref{theorem: the flow space for relatively Anosov}), these ingredients give the continuity of the drift ratios coming from the almost subadditive sequences induced by the cocycle $\msf c$.
We then show that these ratios determine the one sided directional derivatives of the Manhattan manifold parametrization, so their continuity implies the $C^1$-regularity of $\Theta$.

Even though we only discuss relatively Anosov groups, we hope this approach works in more general settings. We prove our results in the more axiomatic frameworks of \Cref{theorem: exact dimensionality framework} and \Cref{theorem: manhattan framework}, and we expect similar axioms to hold in any situation in which a group action has well-behaved Patterson--Sullivan and Bowen--Margulis--Sullivan measures. The key properties we require are a shadow lemma that controls the Patterson--Sullivan measures on balls for some natural premetric, and the finiteness and ergodicity or mixing property for the Bowen--Margulis--Sullivan measures. Promising candidates are the Patterson--Sullivan systems introduced by Kim--Zimmer \cite{kim2025rigidity}, and the GPS systems for convergence group actions \cite{blayac2024counting, blayac2024patterson}.\footnote{Indeed, part of the GPS perspective was adopted by Kim--Oh--Zimmer \cite{kim-zimmer.C1} in their independent proof of \Cref{theoremalpha: regularity of Manhattan manifold}, as they also extend this result to transverse groups.} Recently, the notion of strong positive recurrence (related to finiteness of Bowen--Margulis--Sullivan measures) for GPS systems was developed by Wen \cite{wen}.

\medskip
\subsection{Structure of the paper}
In \Cref{section: preliminaries} we recall the notions of relatively hyperbolic groups and Gromov models, Lie theoretic notation, relatively Anosov groups, Patterson–Sullivan measures, and the flow spaces constructed by Kim--Oh--Wang \cite{kim2025properly} and further studied by Kim--Oh \cite{kim2025relatively}. In \Cref{section: first return} we construct the first return systems and group-valued cocycles used throughout the paper. In \Cref{section: exact dimensionality} we first prove \Cref{theorem: exact dimensionality framework}, a general theorem for exact dimensionality, which we then apply to the Patterson--Sullivan measures and the visual quasi-metrics arising from Gromov models for relatively Anosov groups and complete the proof of \Cref{theoremalpha: exact dimensionality}. In \Cref{section: hyperbolicity of relatively Morse} we prove \Cref{theoremalpha: relativemorse implies hyperbolic} that for relatively Morse groups, the scalar Cartan metrics are hyperbolic. In particular, the premetric constructed in Dey–Kim–Oh \cite{dey2024ahlfors} is a visual quasi-metric for some Gromov model, and as a consequence we deduce \Cref{corollaryalpha: premetric exact dimensionality}. Finally, in \Cref{section: manhattan manifolds} we prove \Cref{theorem: manhattan framework}, a general theorem for regularity of Manhattan curves. This is applied to relatively Anosov groups to prove \Cref{theoremalpha: regularity of Manhattan manifold} and deduce \Cref{corollaryalpha: strict concavity of growth indicator}, the strict concavity of the growth indicators.

\medskip
\subsection{Acknowledgments}

The authors thank Subhadip Dey, Dongryul Kim, D\'idac Mart\'inez-Granado, Hee Oh, Zhufeng Yao and Andrew Zimmer for helpful discussions and comments related to this work.

This paper is based upon work supported by the National Science Foundation under Grant No. DMS-2424139, while the authors were in residence at the Simons Laufer Mathematical Sciences Institute in Berkeley, California, during the Spring 2026 semester. E. Reyes is supported by ANID Fondecyt Iniciaci\'on grant 11260637.

\section{Preliminaries}\label{section: preliminaries}

\subsection{Hyperbolic geometry}

\subsubsection{Relatively hyperbolic groups}\label{section: relatively hyperbolic groups}

The notion of relatively hyperbolic group was introduced by Gromov \cite{gromov1987hyperbolic}. For the preliminaries needed below, we refer the reader to Bowditch \cite{bowditch2012relatively} and Hruska \cite{hruska2010relative}.

Suppose $\Gamma$ is a finitely generated group and $\Pc$ is a finite collection of subgroups of $\Gamma$. We say $\Gamma$ is \emph{hyperbolic relative to} $\Pc$ if it acts as a geometrically finite convergence group on a compact metrizable space $\partial(\Gamma,\Pc)$ with the stabilizers of parabolic points precisely the conjugates of subgroups in $\Pc$. In this case we call $(\Gamma,\Pc)$ a \emph{relatively hyperbolic group} and $\partial(\Gamma,\Pc)$ the \emph{Bowditch boundary} of $(\Gamma,\Pc)$. Following the work of \cite{yaman2004topological} and \cite{bowditch2012relatively}, $(\Gamma,\Pc)$ is a relatively hyperbolic group if and only if it admits a cusp uniform action on a Gromov hyperbolic space $Y$ with cusp subgroups being the conjugates of subgroups in $\Pc$, in which case $\partial(\Gamma,\Pc)$ is $\Gamma$-equivariantly identified with the Gromov boundary $\partial Y$ of $Y$. Such a $Y$ is called a \emph{Gromov model} of $(\Gamma,\Pc)$. Roughly speaking, this means that $\Gamma$ acts properly discontinuously by isometries on $Y$ and $Y$ admits a decomposition into a cocompact part and finitely many orbits of horoballs. 

\begin{remark}\label{remark: gromov models}
    In the context of this paper, we require the Gromov model to be a \emph{roughly geodesic space} rather than geodesic space, meaning that any two points can be joined by a $(1,c)$-quasi-geodesic for some uniform $c\geqslant 0$. This assumption does not alter the expected properties for Gromov models.
    In addition, we always assume that the Gromov models are \emph{taut}, meaning that any point lies in the neighborhood of a bi-infinite rough geodesic of uniform radius, which is a standard assumption for Gromov models (cf.~\cite[Definition~3.10]{kapovich2018relativizing}).
\end{remark}

Suppose $(\Gamma,\Pc)$ is a relatively hyperbolic group. The work of Groves--Manning \cite{groves2008dehn} provides a standard construction of a Gromov model for $(\Gamma,\Pc)$, which we denote by $X_{\GM}=X_{\GM}(\Gamma,\Pc)$. This space is a locally finite graph, whose cocompact part can be identified with a Cayley graph of $\Gamma$ for some finite generating set. 

Let $\Gc(X_{\GM})$ denote the space of bi-infinite geodesics in $X_{\GM}$ parametrized by length, endowed with the geodesic flow $\varphi^s\sigma(\cdot)=\sigma(\cdot+s)$. We define a metric on $\Gc(X_{\GM})$ by \[\msf d_{\Gc(X_{\GM})}(\sigma_1,\sigma_2)= \int_{\Rb} e^{-\frac{\abs{t}}{2}} \msf{d}_{\GM}(\sigma_1(t),\sigma_2(t)) \de t ~.\] 
There is a natural isometric action of $\Gamma$ on $\Gc(X_{\GM})$, so that the coarse projection $\Gc(X_{\GM}) \to X_{\GM}$, $\sigma \mapsto \sigma(0)$ is a $\Gamma$-equivariant quasi-isometry and $s\mapsto \varphi^{s}\sigma$ is a geodesic parametrized by length for any $\sigma\in \Gc(X_{\GM})$.\\

\subsubsection{Visual quasi-metrics and shadows}\label{section: visual metrics and shadows}
Let $(Y,\msf d_Y)$ be a $\delta$-hyperbolic space with a basepoint $y_0$, and let $(\cdot,\cdot)_{y_0}$ denote the corresponding Gromov product, which we also extend to the Gromov boundary. Let $\vis{d}_Y$ be the visual quasi-metric on $\partial Y$ given by
\[\vis{d}_Y(\xi,\eta)= \begin{cases}e^{-(\xi,\eta)_{y_0}} & \text{if } \xi\neq \eta~,\\ 0 & \text{if } \xi=\eta~.\end{cases}\] 
Actually, it is well-known (see for example \cite[Part~III, Proposition~3.21]{bridson2013metric}) that for $\ep>0$ small enough there exists $C\geqslant 1$ and a metric $\vis{d}_{\ep,Y}$ on $\partial Y$ such that \[C^{-1}e^{-\epsilon(\xi,\eta)_{y_0}} \leqslant \vis{d}_{\ep,Y}(\xi,\eta) \leqslant e^{-\epsilon(\xi,\eta)_{y_0}} \quad \text{ for all } \xi,\eta\in \partial Y~.\] Note that all of our discussion throughout the paper for $\vis{d}_Y$ can be adapted to $\vis{d}_{\epsilon,Y}$.

The $R$-\emph{shadow} of $y$ viewed from $y_0$ is defined as \[O_R^Y(y_0,y) = \{\xi\in \partial Y: (\xi,y)_{y_0} \geqslant d_Y(y_0,y)-R\}~.\] The lemma below follows from a standard hyperbolic geometry argument (cf.~ \cite[Proposition~2.1]{BHM.harmonic}). For $r>0$ and $\xi\in \partial Y$, let $B_{\vis{d}_Y}(\xi,r)=\{\eta\in \partial Y: \vis{d}_Y(\xi,\eta)< r \}$ be the ball of radius $r$ centered at $\xi$.

\begin{lemma}\label{lemma: shadow and visual ball comparison}
    For any $R_0\geqslant 0$, there exists $R\geqslant R_0$ and $s \geqslant 0$ such that for every $y \in Y$ and every $\zeta\in O_{R_0}^Y(y_0,y)$, \[B_{\vis{d}_Y}(\zeta, e^{-\msf d_Y(y_0,y) - s}) \subset O_{R}^Y(y_0,y) \subset B_{\vis{d}_Y}(\zeta, e^{-\msf d_Y(y_0,y) + s})~.\]
\end{lemma}

We also need the following comparison for shadows defined from different Gromov models.

\begin{lemma}\label{lemma: shadows comparison}
    Suppose $(\Gamma,\Pc)$ is a relatively hyperbolic group, $X,Y$ are two Gromov models of $(\Gamma,\Pc)$, and $x_0\in X$ and $y_0\in Y$ are fixed basepoints. Denote by $f:\partial X \to \partial Y$ the $\Gamma$-equivariant homeomorphism via the identifications of $\partial X$ and $\partial Y$ to $\partial(\Gamma,\Pc)$. Then
    \begin{enumerate}
        \item For any $r>0$ there exists $r'>0$ such that for every $\gamma\in \Gamma$ we have
        \[f(O_r^{X}(x_0,\gamma x_0))\subset O_{r'}^Y(y_0,\gamma y_0)~.\]
        \item For any $r'>0$ there exists $r>0$ such that for every $\gamma\in \Gamma$ we have
        \[O_{r'}^Y(y_0,\gamma y_0)\subset f(O_r^{X}(x_0,\gamma x_0))~.\]
    \end{enumerate} 
\end{lemma}

\begin{proof}
    This essentially follows from \cite[Proposition~5.5]{blayac2024patterson} as we can compare the shadows in both $\partial X$ and $\partial Y$ with balls using compatible metrics in the sense of \cite[Definition~2.2]{blayac2024patterson}. Then the lemma follows from the uniqueness of the compactifying topology \cite[Proposition~2.3]{blayac2024patterson}. 
\end{proof}

\subsection{Notations on Lie groups}\label{section: Lie groups}

We mostly follow the same setting on Lie theory as in \cite[Section~2, Section~5]{kim2025relatively} and we keep this notation throughout the paper.

Let $\msf G$ be a connected semisimple real algebraic group with finite center and identity element $e$. Fix a Cartan decomposition $\msf G=\msf K\exp(\fraka^+)\msf K$, where $\msf K<\msf G$ is a maximal compact subgroup and $\fraka^+\subset \fraka$ is a closed positive Weyl chamber. We denote the Cartan projection by $\mu:\msf G\to \fraka^+$. Let $\Pi$ be the set of simple roots determined by $\fraka^+$. Given a non-empty subset $\theta\subset \Pi$, the associated \emph{$\theta$-Cartan space} is the subspace \[\fraka_\theta=\bigcap_{\alpha\in \Pi\setminus \theta}\ker \alpha~.\] We write $p_\theta:\fraka\to\fraka_\theta$ for the natural projection, and for $g\in \msf G$ we denote $\mu_\theta(g)=p_\theta(\mu(g))$. Let $\msf  P< \msf G$ be the standard minimal parabolic subgroup corresponding to $\Pi$ and $\Fc = \msf G/\msf P$ be the flag manifold. Similarly, let $\msf P_\theta<\msf G$ be the standard parabolic subgroup corresponding to $\theta$ with $\theta$-flag manifold  $\Fc_\theta=\msf G/\msf P_\theta$. We also let $\pi_\theta: \Fc \to \Fc_\theta$ be the canonical projection.

We write $i=-\Ad(w_0)$ for the opposition involution, where $w_0$, from the normalizer of $\exp(\fraka)$, is a representative of the longest Weyl element. We say that two flags $\xi \in \Fc_\theta$ and $\eta\in \Fc_{i(\theta)}$ are \emph{transverse} if there exists $g\in \msf G$ such that $\xi = g \msf P_\theta$ and $\eta = g w_0 \msf P_{i(\theta)}$. We denote \[\Fc_\theta^{(2)} = \{(\xi,\eta)\in \Fc_\theta\times \Fc_{i(\theta)}: \xi \text{ and } \eta \text{ are transverse}\}~.\] 

Given $\xi_0\in \Fc$, let $k\in \msf K$ satisfy $\xi_0 = k \msf P$, and for $g\in \msf G$ we set $\sigma(g^{-1},\xi_0) \in \fraka$ to be the unique element such that $g^{-1}k \in \msf K \exp(\sigma(g^{-1},\xi_0)) \msf N$ via the Iwasawa decomposition $\msf G = \msf K \msf P = \msf K \exp(\fraka) \msf N$. Define the \emph{$\fraka_\theta$-valued Busemann function} $\beta^\theta:\Fc_\theta \times \msf G \times \msf G\to \fraka_\theta$ by \[\beta_\xi^\theta(g,h) = p_\theta(\sigma(g^{-1},\xi_0) - \sigma(h^{-1},\xi_0)) \quad \text{ for all } \xi\in \Fc_\theta,\ g,h\in \msf G~,\] where $\xi_0\in \pi_\theta^{-1}(\xi)$. By \cite[Lemma~6.1]{quint2002mesures}, the value $\beta_\xi^\theta(g,h)$ is independent of the choice of $\xi_0$. The \emph{$\fraka_\theta$-valued Gromov product} $G^\theta:\Fc_\theta^{(2)}\to \fraka_\theta$ is defined by \[G^\theta(\xi,\eta) = \frac{1}{2}(\beta_\xi^\theta(e,g)+i(\beta_\eta^{i(\theta)}(e,g)))\] where $g\in \msf G$ is proximal and such that $(g^+, g^-) = (\xi, \eta)$, and by \cite[Lemma~9.13]{kim2025properly}, $G^\theta(\xi,\eta)$ is independent of the choice of $g$. Here $g^+$ (respectively, $g^-$) denotes the attracting (respectively, repelling) fixed flag of $g$ in $\Fc_\theta$ (respectively, in $\Fc_{i(\theta)}$).

The next lemma follows easily from the $\msf G$-invariant property of $\beta^\theta$.

\begin{lemma}[{\cite[Lemma~5.2]{dey2024ahlfors}}]\label{lemma: Gromov product conformal identity}
    For any $g\in \msf G$ and $(\xi,\eta) \in \Fc_\theta^{(2)}$ we have 
    \[G^\theta(\xi,\eta) = G^\theta(g^{-1}\xi,g^{-1}\eta) + \frac{1}{2}(\beta^\theta_\xi(e,g)+ i(\beta^{i(\theta)}_\eta(e,g)))~.\]
\end{lemma}

Let $o' = \msf K \in \msf G/ \msf K$ be the fixed basepoint. We define the $R$-\emph{shadow} of $q\in \msf G/\msf K$ viewed from $o'$ by \[O^\theta_R(o',q) = \{k \msf P_\theta\in \Fc_\theta: k\in \msf K,\ k \exp(\fraka^+)o' \cap B(q,R) \ne \emptyset\}\subset \Fc_\theta~,\]
where $B$ denotes the closed ball with respect to the $\msf G$-invariant Riemannian metric on $\msf G/\msf K$.

\subsection{Relative Anosov groups}\label{section: preliminary on relative anosov}

Suppose $\Gamma < \msf G$ is a discrete subgroup. We say that $\Gamma$ is \emph{$\theta$-divergent} if for each $\al\in \theta$ we have \[\alpha(\mu(\gamma)) \to +\infty \quad \text{as} \quad \gamma \to \infty~.\] In this case we denote by $\Lambda_\theta$ its $\theta$\emph{-limit set}, defined as the set of all the accumulation points of $\Gamma$ on $\Fc_\theta$.

Suppose $\Pc$ is a collection of subgroups of $\Gamma$ such that $(\Gamma,\Pc)$ is relatively hyperbolic. Following \cite{kapovich2018relativizing} and \cite{zhu2022relatively}, we say that $\Gamma$ is a \emph{$\theta$-Anosov group relative to} $\Pc$, or simply a \emph{relatively $\theta$-Anosov group}, if it is $\theta$-divergent (and hence $i(\theta)$-divergent), $\Lambda_{\theta\cup i(\theta)}$ is transverse, meaning that any two distinct flags in $\Lambda_{\theta\cup i(\theta)}$ are transverse, and there exists a $\Gamma$-equivariant homeomorphism $\partial(\Gamma,\Pc) \to \Lambda_{\theta\cup i(\theta)}$, called the \emph{limit map} of $\Gamma$. Such a limit map is uniquely characterized by the relative Anosov property. We denote by $f:\partial(\Gamma,\Pc) \to \Lambda_\theta$ (respectively $f_i:\partial(\Gamma,\Pc) \to \Lambda_{i(\theta)}$) the composition of the limit map and the canonical projection $\Lambda_{\theta\cup i(\theta)} \to \Lambda_\theta$ (respectively, $\Lambda_{\theta\cup i(\theta)} \to \Lambda_{i(\theta)}$).

The \emph{$\theta$-limit cone} $\Lc_\theta(\Gamma) \subset \fraka_\theta$, introduced by Benoist \cite{benoist1997proprietes}, is the closed cone generated by the asymptotic directions of $\mu_\theta(\Gamma)$. We will need the following lemma.

\begin{lemma}\label{lemma: comparison positive linear forms}
    Let $\psi,\psi' \in \fraka_\theta^\ast$ be two linear forms which are positive on $\Lc_\theta(\Gamma)\setminus\{0\}$. Then there exist constants $A\geqslant 1$ and $B\geqslant 0$ such that \[\psi(\mu_\theta(\gamma))\leqslant A\  \psi'(\mu_\theta(\gamma))+B\quad \text{ for all } \gamma\in \Gamma~.\]
\end{lemma}

\begin{proof}
    We fix a norm on $\fraka_\theta$ and let $S\subset \fraka_\theta$ be the unit sphere with respect to this norm. For any $U \subset S$, we denote by $\Cc(U)$ the cone that projects to $U$. We now pick $U$ to be a compact neighborhood $\wt \Lc$ of $\Lc_\theta(\Gamma) \cap S$ that contains on the interior of $\fraka_\theta^+$ and such that $\psi$ and $\psi'$ are positive on $\Cc(U)$. By the compactness of $U$, there exists a constant $A >0$ such that $\psi(u)\leqslant A \psi'(u)$ for any $u \in U$ and hence for any $u\in \Cc(U)$. By the definition of $\Lc_\theta(\Gamma)$, $\mu_\theta(\gamma)$ is contained in $\wt \Lc$ for all but finitely many $\gamma \in \Gamma$. For the remaining finitely many elements outside $\Cc(U)$, set $B$ to be the maximum of their $\psi(\mu_\theta(\cdot)) - A \psi'(\mu_\theta(\cdot))$ values, which completes the proof.
\end{proof}

\subsection{Patterson--Sullivan measures}\label{section: Patterson--Sullivan measures}

Suppose $\Gamma< \msf G$ is a $\theta$-divergent subgroup. The following $\theta$-growth indicator is introduced in \cite{kim2025properly}, generalizing that from \cite{quint2002divergence}. Let $\norm{\cdot}$ be a fixed norm on $\fraka_\theta$. For an open cone $\Cc\subset \fraka_\theta$, define \[h_\Gamma^\theta(\Cc)=\limsup_{T\to\infty} \frac{1}{T} \log \#\{\gamma\in\Gamma:\mu_\theta(\gamma)\in \Cc,\ \norm{\mu_\theta(\gamma)}\leqslant T\}~.\] The \emph{$\theta$-growth indicator} of $\Gamma$ is the homogeneous function given by \[\psi_\Gamma^\theta:\fraka_\theta\to \Rb\cup\{-\infty\}~, \quad v \mapsto \norm{v} \inf_{\Cc\ni v} h_\Gamma^\theta(\Cc)~,\] for which we also set $\psi_\Gamma^\theta(0)=0$. The function $\psi_\Gamma^\theta$ is homogeneous and concave on $\Lc_\theta(\Gamma)$.

\begin{definition}
    We say a linear form $\psi\in \fraka_\theta^\ast$ is
    \begin{enumerate}
        \item $(\Gamma,\theta)$-\emph{proper}, if for any $C>0$ there are at most finitely many $\gamma\in \Gamma$ satisfying $\psi(\mu_\theta(\gamma))\leqslant C$.
        \item \emph{positive}, if it is positive on $\Lc_\theta(\Gamma)\setminus \{0\}$;
        \item \emph{tangent to the $\theta$-growth indicator of} $\Gamma$, if $\psi(\vv)\geqslant \psi_\Gamma^\theta(\vv)$ for any $\vv\in \fraka_\theta$ and there exists some non-zero $\vv_0\in \Lc_\theta(\Gamma)$ such that $\psi(\vv_0)=\psi_\Gamma^\theta(\vv_0)$.
    \end{enumerate}
\end{definition}

For $\psi\in \fraka_\theta^\ast$, define the $\psi$-Poincar\'e series \[Q^\psi_\Gamma(s) = \sum_{\gamma\in\Gamma} e^{-s\psi(\mu_\theta(\gamma))}\] and the corresponding \emph{critical exponent} \[\delta_\psi(\Gamma)=\inf\{s>0:Q^\psi_\Gamma(s)<+\infty\}~.\] With the normalization above, $\psi$ is tangent to $\psi_\Gamma^\theta$ if and only if $\delta_\psi(\Gamma)=1$ \cite[Corollary~4.6]{kim2025properly}.

\begin{definition}\label{definition: Patterson--Sullivan measures}
    Given $\psi \in \fraka_\theta^\ast$, a \emph{$(\Gamma,\psi)$-conformal measure} is a Borel probability measure $\nu$ on $\Fc_\theta$ such that \[\frac{\de \gam_\ast \nu}{\de \nu}(\xi)=e^{\psi(\beta_\xi^\theta(e,\gamma))} \quad \text{ for all } \gamma \in \Gamma \text{ and } \nu \text{-almost every } \xi\in \Fc_\theta~.\] A \emph{$(\Gamma,\psi)$-Patterson--Sullivan} measure is a $(\Gamma, \psi)$-conformal measure supported on $\Lam_\theta$.
\end{definition}

Now we suppose that $\Gamma<\msf G$ is a $\theta$-Anosov group relative to a collection $\Pc$ of subgroups.

\begin{theorem}[{\cite[Theorem~1.1, Corollary~1.5]{canary2025patterson}}]\label{theorem: PS existence uniqueness ergodicity}  Let $\psi\in \fraka_\theta^\ast$ be $(\Gamma,\theta)$-proper, positive, and tangent to the $\theta$-growth indicator of $\Gamma$. Then the following hold.
    \begin{enumerate}
        \item The $\psi$-Poincar\'e series diverges at its critical exponent, i.e., $Q_\Gamma^\psi(1) = +\infty$.
        \item There exists a unique $(\Gamma,\psi)$-Patterson--Sullivan probability measure on $\Lambda_\theta$. We denote it by $\nu_\psi$.
    \end{enumerate}
\end{theorem}

For $\psi$ as in the previous theorem, we denote by $\wh \nu_\psi$ the unique $(\Gamma,\psi\circ i)$-Patterson--Sullivan probability measure on $\Lambda_{i(\theta)}$.

In what follows, we also need the following shadow comparisons. Recall that $f:\partial X_{\GM}\ra \Fc_\theta$ denotes the limit map corresponding to a Groves--Manning cusped space, and for a fixed basepoint $o\in X_{\GM}$ we denote $O_R^{\GM}(o,x)=O_R^{X_{\GM}}(o,x)$.

\begin{lemma}[{\cite[Proposition~5.7]{kim2025relatively}}]\label{lemma: the comparison for relatively Anosov}
    For any sufficiently large $R > 0$, there exist $r_1,r_2>0$ such that for every $\gamma\in\Gamma$, we have \[O^\theta_{r_1}(o',\gamma o')\cap \Lambda_\theta \subset f(O^{\GM}_R(o,\gamma o)) \subset O^\theta_{r_2}(o',\gamma o')\cap \Lambda_\theta~.\] Moreover, we can choose $r_1\to\infty$ as $R\to\infty$.
\end{lemma}

\begin{lemma}[{\cite[Lemma~9.2]{kim2025relatively}}]\label{lemma: the shadow lemma for relatively Anosov}
    For any sufficiently large $R > 0$, there exists a constant $C\geqslant 1$ such that for every $\gamma\in\Gamma$ we have \[C^{-1} e^{-\psi(\mu_\theta(\gamma))} \leqslant \nu_\psi(O^\theta_R(o',\gamma o')) \leqslant C e^{-\psi(\mu_\theta(\gamma))}~.\]
\end{lemma}

\subsection{Reparametrization cocycles}\label{section: preliminary on fibered dynamical systems}

Let $\Gamma<\msf G$ be a $\theta$-Anosov group relative to a collection $\Pc$ of subgroups and let $X_{\GM}$ be a Groves--Manning cusp space of $(\Gamma,\Pc)$. Suppose $\psi\in \fraka_\theta^\ast$ is a $(\Gamma,\theta)$-proper, positive linear form that is tangent to the $\theta$-growth indicator of $\Gamma$. The work of Kim--Oh \cite{kim2025relatively} considers another flow space for the study of Patterson--Sullivan theory and the following result presents its relation to $\Gc(X_{\GM})$.

We write $\Lambda_\theta^{(2)} = (\Lambda_\theta\times \Lambda_{i(\theta)}) \cap \Fc_\theta^{(2)}$. The topological space
$\widetilde\Omega_\psi= \Lambda_\theta^{(2)}\times \Rb$ is endowed with a properly discontinuous action of $\Gamma$ by the fibered action associated to the $\psi$-value on the Busemann cocycle, see \cite[Section~9]{kim2025properly}. More precisely,  for any $\gamma\in \Gamma$ and $(\xi,\eta,t)\in \widetilde\Omega_\psi$ we have
\[\gamma (\xi,\eta,t) = (\gamma \xi, \gamma \eta, t + \psi(\beta^\theta_\xi (\gamma^{-1},e)))~.\]
We denote the quotient by $\Omega_\psi=\Gamma\backslash \widetilde\Omega_\psi$ and the quotient map by $p_\psi:\widetilde\Omega_\psi\to \Omega_\psi$. By \cite[Theorem~9.2]{kim2025properly}, $\Omega_\psi$ is locally compact Hausdorff. The translation flow $\phi^s(\xi,\eta,r)=(\xi,\eta,r+s)$ on $\wt \Omega_\psi$ commutes with the action of $\Gamma$ and therefore descends to a flow on $\Omega_\psi$, again denoted by $\phi^s$.
For $v = (\xi,\eta,s)\in \wt \Omega_{\psi}$, we write $v^+ = \xi$ and $v^- = \eta$.

The next theorem is due to Kim--Oh is one of our main tools used throughout the paper. 

\begin{theorem}[{\cite[Theorem~1.4, Lemma~6.5]{kim2025relatively}}]\label{theorem: the flow space for relatively Anosov}
    There exists a continuous, surjective, proper, $\Gamma$-equivariant map $\wt \Psi:\Gc(X_{\GM})\to \wt\Omega_\psi$ with a continuous cocycle $\wt \sft:\Gc(X_{\GM})\times \Rb\to \Rb$ such that for all $\sigma\in \Gc(X_{\GM})$ and $s\in \Rb$,
    \begin{enumerate}
        \item $\wt\Psi(\varphi^s\sigma) = \phi^{\wt \sft(\sigma,s)}(\wt\Psi(\sigma))$.
        \item $\wt \sft(\sigma,s) = -\wt \sft(\varphi^s\sigma,-s)$.
        \item There are constants $a,a',B>0$ such that \[a\abs{s}-B \leqslant \wt \sft(\sigma,\abs{s}) \leqslant a'\abs{s}+B \quad \text{ for all } \sigma\in\Gc(X_{\GM}),\ s\in\Rb~.\]
        \item The diameter of \[\{\sigma(0)\in X_{\GM} : \sigma \in \wt \Psi^{-1}(v)\}\] is uniformly bounded for all $v\in \wt \Omega_\psi$.
        \item If $Q\subset X_{\GM}$ is a compact subset, then there is a constant $C_Q\geqslant0$ such that whenever $\sigma\in \Gc(X_{\GM})$, $t \geqslant 0$ and $\gamma \in \Gamma$ are such that $\sigma(0), \gamma^{-1}\sigma(t) \in Q$, we have $\abs{\wt \sft(\sigma,t) - \psi(\mu_\theta (\gamma))} \leqslant C_Q$.
    \end{enumerate} 
\end{theorem}

The theorem above also implies that $\wt\Psi$ restricts to a surjection from $\{\varphi^r(\sig):r\in \R\}$ onto $\{\phi^s(\wt\Psi(\sig)):s\in \R\}$ for any $\sig\in \Gc(X_{\GM})$. Since $\wt\Psi$ is surjective, we can define an inverse cocycle of $\wt \sft$ as the map $\wh \sft: \wt \Omega_\psi \times \R \to \R$ given by \[\wh \sft(v,s)=\inf\{r\in \R: \wt\Psi(\varphi^r(\sig))=\phi^s(v)\text{ for some }\sig\in \wt\Psi^{-1}(v)\}~.\] 

\begin{lemma}\label{lemma: hat t almost cocycle}
    $\wh \sft$ is a measurable almost cocycle. That is, there exists $C>0$ such that for all $v\in \wt \Omega_\psi$ and $s,s'\in \R$ we have \[\abs{\wh \sft(v,s+s') -\wh \sft(v,s) -\wh \sft(\phi^s(v),s')}\leqslant C~.\] 
    Moreover, $\wh {\msf t}$ is $\Gamma$-invariant, so it descends to a measurable almost cocycle $\ov \sft: \Omega_\psi\times \R \ra \R$. 
\end{lemma}

\begin{proof}
   Since $\wt \sft$ is a cocycle, the almost cocycle property follows after noting that the sets 
  $\{r\in \R: \wt\Psi(\varphi^r(\sig))=\phi^s(v)\text{ for some }\sig\in \wt\Psi^{-1}(v)\}$ are uniformly bounded for $v\in \wt\Omega_\psi$ and $s\in \R$. This last property follows from \Cref{theorem: the flow space for relatively Anosov}~(4). The $\Gamma$-invariance of $\wh \sft$ follows from the $\Gamma$-invariance of $\wt \sft$.
\end{proof}

The \emph{Bowen--Margulis current} associated to $\psi$ is the measure $\sfL_\psi$ on $\Lam_\theta^{(2)}$ defined as \[\de \sfL_\psi(\xi,\eta) =e^{2\psi(G^\theta(\xi,\eta))} \de\nu_\psi(\xi)\de\wh\nu_\psi(\eta)~.\]
By the conformality of $\nu_\psi$ and $\wh\nu_\psi$, and by \Cref{lemma: Gromov product conformal identity}, $\sfL_\psi$ is $\Gamma$-invariant. The (lifted) \emph{Bowen--Margulis--Sullivan measure} on $\widetilde\Omega_\psi$ is defined as 
\[ \widetilde{\frakm}_\psi =  \sfL_\psi \otimes \mathrm{Leb}_\R~,\]
where $\mathrm{Leb}_\R$ denotes the Lebesgue measure on $\Rb$. This measure descends to a $\phi$-invariant Radon measure $\frakm_\psi$ on $\Omega_\psi$, which we again call the \emph{Bowen--Margulis--Sullivan measure} on $\Omega_\psi$.

One of the main results from \cite{kim2025relatively} is the following. 

\begin{theorem}[{\cite[Theorem~1.1]{kim2025relatively}, \cite[Theorem~8.1, Section~10]{blayac2024counting}}]\label{theorem: BMS finite and strongly mixing}
   The measure $\frakm_\psi$ is finite and $(\Omega_\psi, \frakm_\psi,\phi^t)$ is strongly mixing.
\end{theorem}

\section{First return maps and induced cocycles}\label{section: first return}

In this section we construct the general framework for first return maps associated to a given dynamical system. Similar ideas have been used to study flow spaces associated to hyperbolic groups in \cite[Section~3.3]{EPS} and \cite[Section~6]{CMGR}.

Suppose $\Omega$ is a measurable space and $\phi=(\phi^t)_{t\in \Rb}$ is a measurable flow on $\Omega$. Given a measurable subset $K\subset \Omega$, its \emph{return set} is \[\Fc_K = \{ v \in K : \phi^n(v)\in K \text{ for infinitely many } n\in \Nb\}\] and the \emph{first return map} is $\Phi :\Fc_K \to \Fc_K$ given by \[\Phi (v)=\phi^{T_1(v)}(v), \quad \text{ where }\quad T_1(v)=\min\{n\in \N:\phi^n(v)\in K\}~.\] We call $T_1(v)$ the \emph{first return time}, and let $T_n(v)=T_1(\Phi^{n-1}v)+T_{n-1}(v)$ be the \emph{$n$-th return time} defined inductively.
Note that $\Phi$ is actually the first return of $\phi^1$ rather than the continuous flow.
By Poincar\'e recurrence, for any $\phi$-invariant probability measure $\frakm$ on $\Omega$ with $\frakm(K)>0$, we have $\frakm(\Fc_K)=\frakm(K)$. For this reason, we often write that the return map $\Phi$ and the return times $T_n$ are defined ($\frakm$-almost everywhere) on $K$. It follows that the probability measure \[\wh \frakm^\Phi = \frac{1}{\frakm(K)}\frakm\vert_{K}\] on $K$ is $\Phi$-invariant. We call $(\Fc_K, \wh \frakm^\Phi, \Phi)$ the \emph{first return dynamical system} associated to $(\Omega,\frakm,\phi)$ and $K$.

The following observation about first return dynamical systems is straightforward.

\begin{lemma}\label{lemma: mixing implies ergodic first return}
    If $(\Omega,\frakm,\phi)$ is mixing, then $(\Fc_K,\wh \frakm^\Phi,\Phi)$ is ergodic. \qed
\end{lemma}

Consider now another measurable space  $\wt \Omega$ with a measurable flow that we also denote by $\phi$. Let $\Gamma$ be a countable group acting measurably on $\wt \Omega$ and commuting with the flow, and suppose there is a $\Gamma$-invariant measurable map $p:\wt \Omega \to \Omega$ such that $\phi^t \circ p=p\circ \phi^t$ for all $t\in \Rb$. Suppose further that a measurable set $\wh K \subset \wt \Omega$ satisfies

\begin{enumerate}
    \item  $\gamma \wh K \cap \wh K = \emptyset$ for all $\gamma\neq e$ in $\Gamma$ and
    \item  the restriction $p:\wh K \to K=p(\wh K)$ is a measurable isomorphism.
\end{enumerate}

Under these assumptions, we can define a cocycle $\msf c: \Nb \times \Fc_K  \to \Gamma$ as follows. Given $v\in \Fc_K$, let $\wh v$ be its unique lift in $\wh K$ under $p$. Then for each $n\in \N$ there exists a unique group element $\msf c_n(v)=\msf c(v,n)\in \Gamma$ such that $\phi^{T_n(v)}(\wh v)\in \msf c_n(v)\wh K$. It easily follows that the maps $v\mapsto \msf c_n(v)$ are measurable and satisfy the cocycle relation \[\msf c_{m+n}(v) = \msf c_n(v)\ \msf c_m(\Phi^n(v)) \quad \text{ for all } v\in \Fc_K \text{ and }m,n\in \Nb~.\]
As above, for a $\phi$-invariant probability measure $\frakm$ on $\Omega$ with $\frakm(K)>0$, we have that each $\msf c_n$ is defined (almost everywhere) on $K$ and that the cocycle relation holds almost everywhere.

For most applications, $\wt\Omega$ is a locally compact and metrizable space on which $\Gamma$ acts continuously and properly discontinuously, with $p:\wt\Omega \to \Omega=\Gamma \backslash \wt\Omega$ being the quotient map. In this case $\Omega$ is also locally compact and metrizable. We also assume that the flows $\phi$ on $\wt \Omega$ and $\Omega$ are continuous, hence any $\phi$-invariant and $\Gamma$-invariant Radon measure $\wt\frakm$ on $\wt\Omega$ descends to $\phi$-invariant Radon measure $\frakm$ on $\Omega$. 

In this setting, the next lemma provides sufficient conditions for regularity of the first return map. 

\begin{lemma}\label{lemma: continuous almost everywhere}
For the systems $(\wt\Omega,\phi,\wt\frakm)$ and  $(\Omega,\phi,\frakm)$ as above, suppose that $\frakm$ is a finite measure, and let $\wh K \subset \wt\Omega$ be a compact set with $\wt\frakm(\wh K)>0$ and satisfying: 
\begin{enumerate}
    \item \label{itm:c1} $\gamma \wh K \cap \wh K=\emptyset$ for all $\gamma\neq e$ in $\Gamma$; and,
    \item \label{itm:c2} $\wt\frakm(\partial \wh K)=0$.
\end{enumerate}
Let $K=p(\wh K)$ and let $\msf c: \N \times K \to \Gamma$ be the cocycle induced by $\wh K$ and the first return map of $K$ (defined $\frakm$-almost everywhere). Then for each $n\in \N$, the map $v\mapsto \msf c_n(v)$ is continuous $\frakm$-almost everywhere.
\end{lemma}

\begin{proof}
Our assumptions imply that $K$ is compact with $\frakm(\partial K)=0$. For a fixed $n\in \N$, we claim that the discontinuity points of $\msf c_n$ lie on the union $A=\bigcup_{j=0}^{\infty}{\Phi^{-j}(\partial K)}$, which is $\frakm$-null by condition (2) and the $\Phi$-invariance of $\frakm\vert_K$. Consider a point $v\in \Fc_K \bs A$ with lift $\wh v\in \wh K$, and let $J=\{T_1(v),\dots,T_n(v)\}$. Since $v\notin A$, for $1\leqslant j \leqslant T_n(v)$ we have $\phi^j(v)\in \mathrm{Int}(K)$ if $j\in J$ (here $\mathrm{Int}$ denotes topological interior in $\Omega$), and $\phi^j(v)\in \Omega \bs K$ if $j\notin J$. Therefore, by the continuity of $\phi$, the lift $\wh v' \in \wh K$ of any $v'\in \Fc_K$ close enough to $v$ satisfies $\phi^j(\wh v')\in \msf c_j(v)\wh K$ for $j\in J$ and $\phi^j(v')\notin K$ for $j\in \{1,2,\dots, T_n(v)\}\bs J$. This yields $\Phi^j(v')=\Phi^j(v)$ for $1\leqslant j \leqslant n$ and $\msf c_{n}(v')=\msf c_n(v)$. Hence, $\msf c_n$ is continuous (in fact, locally constant) at $v$.
\end{proof}

In practice, we construct sets $\wh K$ satisfying the above criterion for systems with transverse sections. For this, we keep the continuity assumptions on the systems $(\wt \Omega,\phi,\frakm)$ and $(\Omega,\phi,\frakm)$ as in the previous lemma. We additionally assume that $\wt \Omega =\Cc \times \Rb$ for $\Cc$ a locally compact metrizable space, that the flow $\phi$ on $\wt \Omega$ is given by $\phi^t(x,s)=(x,s+t)$ for all $x\in \Cc,s,t\in \Rb$, that the action of $\Gamma$ on $\wt\Omega$ commutes with $\phi$, and that we have a decomposition $\wt\frakm=\sfL \otimes \mathrm{Leb}_{\Rb}$ with $\sfL$ a Radon measure on $\Cc$. Under these assumptions we have the following.

\begin{lemma}\label{lemma: good local section}
    Let $\wt\Omega=\Cc\times \Rb$ and $\wt \frakm=\sfL\otimes \mathrm{Leb}_{\Rb}$ and $\frakm$ be as above, and let $v\in \wt\Omega$ be a point in the support of $\wt \frakm$. Then for any neighborhood $U$ of $v$ in $\wt\Omega$ there exists a compact set $\wh K\subset U$ containing $v$ and such that $\wt \frakm(\wh K)>0$ and $\wt \frakm(\partial \wh K)=0$.
\end{lemma}

\begin{proof}
    Write $v=(c,s)\in \Cc \times \Rb$, so that $c$ belongs to the support of $\sfL$. For a metric compatible with the topology on $\Cc$, let $B_r$ and $S_r$ be the closed ball and sphere of radius $r$ around $c$, respectively. Then $\sfL(B_r)>0$ for all $r>0$. Since $\Cc$ is locally compact and $\sfL$ is Radon, there are infinitely many $r$, arbitrarily close to $0$, such that $\sfL(S_r)=0$. We pick such a sufficiently small value of $r$, as well as $\ep>0$ small enough such that the set $\wh K=B_r\times [s-\ep,s+\ep]$ is contained in $U$. By construction we have that $\wh K$ is compact and $\wt \frakm(\wh K)>0$, so we are only left to show that $\wt \frakm(\partial \wh K)=0$. This last fact follows from the identity $\partial \wh K=\partial B_r \times [s-\ep,s+\ep] \cup B_r\times \{s-\ep,s+\ep\}$ and the inclusion $\partial B_r\subset S_r$.
\end{proof}

\section{Exact dimensionality}\label{section: exact dimensionality}

In this section we prove \Cref{theoremalpha: exact dimensionality}, asserting the exact dimensionality of Patterson--Sullivan measures for relatively Anosov groups, when seen as measures on boundaries for Gromov models.
We do this by first proving \Cref{theorem: exact dimensionality framework}, a general theorem about exact dimensionality of measures associated to certain group actions.

\subsection{A framework theorem for exact dimensionality}\label{section: exact dimensionality framework}

Our main criterion for exact dimensionality is the following.

\begin{theorem}\label{theorem: exact dimensionality framework}
    Suppose $M$ is a measurable space, $\Gamma$ is a group acting on $M$ by measurable isomorphisms, $d$ is a premetric on $M$ such that all balls \[B_d(x,r)=\{y\in M:d(x,y)<r\}\] are measurable, and $\nu$ is a probability measure on $M$.   Assume the following data.
    \begin{enumerate}
        \item A measurable dynamical system $(\Fc,\frakm,\Phi)$, where $\Phi:\Fc \to \Fc$ is measurable and $\frakm$ is a $\Phi$-invariant probability measure.
        \item A measurable map $\pi:\Fc \to M$ such that the image of a $\frakm$-full measure set for $\Fc$ is a measurable set of $\nu$-full measure in $M$.
        \item For each $\gamma\in \Gamma$, a measurable set $O(\gamma)\subset M$, together with a constant $C\geqslant 1$ and a function $ \alpha:\Gamma\to \mathbb R_{\geqslant 0}$ such that \[C^{-1}e^{-\alpha(\gamma)}\leqslant \nu(O(\gamma))\leqslant Ce^{-\alpha(\gamma)}\quad \text{ for all }\gamma\in \Gamma~.\]
        \item A measurable cocycle $\msf c:\Nb\times \Fc\to \Gamma$, defined $\frakm$-almost everywhere, such that \[\msf c_{m+n}(v) = \msf c_n(v)\ \msf c_m(\Phi^n(v))~, \quad \text{ for all } m,n\in \Nb \text{ and }\frakm\text{-almost all }v\in \Fc~.\]
        \item A function $\beta:\Gamma \to \R_{\geqslant 0}$ and a constant $s \geqslant 0$ such that for $\frakm$-almost every $v\in \Fc$ and all $n$ large enough (depending on $v$) we have \[B_d(\pi(v),e^{-\beta(\msf c_n(v))-s})\subset O(\msf c_n(v))\subset B_d(\pi(v),e^{-\beta(\msf c_n(v))+s})~.\]
        \item Constants  $\lambda_\alpha \geqslant 0$ and $\lambda_\beta>0$ such that for $\frakm$-almost every $v\in \Fc$ we have \[\frac{\alpha(\msf c_n(v))}{n}\to \lambda_\alpha \quad \text{and} \quad \frac{\beta(\msf c_n(v))}{n}\to \lambda_\beta \quad \text{as }n\to \infty~.\]
    \end{enumerate}
    Then for $\nu$-almost every $x\in M$ we have \[\lim_{r\to 0}\frac{\log \nu(B_d(x,r))}{\log r} = \frac{\lambda_\alpha}{\lambda_\beta}~.\] That is, $\nu$ is exact dimensional with respect to $d$.
\end{theorem}

\begin{proof}
    Let $E\subset \Fc$ be a measurable set of full $\frakm$-measure on which all the conditions described above hold simultaneously for every $n\in \Nb$. By condition (2), the set $\pi(E)$ is measurable and has full $\nu$-measure. It is therefore enough to prove the conclusion for every point of the form $x = \pi(v)$, for some $v\in E$.

    Fix such a $v\in E$, and write $x = \pi(v)$, $\gamma_n = \msf c_n(v)$, $\alpha_n = \alpha(\gamma_n)$, and $\beta_n = \beta(\gamma_n)$. Then condition (6) implies \[\frac{\alpha_n}{n}\to \lambda_\alpha\quad \text{ and }\quad \frac{\beta_n}{n}\to \lambda_\beta>0 \quad \text{ as }n\to \infty.\]
    We write $q_n^-=e^{-(\beta_n+s)}$ and $q_n^+=e^{-(\beta_n-s)}$, and note that $q_n^{\pm}\to 0$ as $n\to \infty$. Then conditions (3) and (5) translate to \[ C^{-1}e^{-\alpha_n}\leqslant \nu(O(\gamma_n))\leqslant Ce^{-\alpha_n}\quad \text{and} \quad B_d(x,q_n^-)\subset O(\gamma_n)\subset B_d(x,q_n^+)~\]
    for all $n$ large enough. Hence $\nu(B_d(x,q_n^-))\leqslant Ce^{-\alpha_n}$ and $\nu(B_d(x,q_n^+))\geqslant C^{-1}e^{-\alpha_n}$.

    Let $r\in (0,1)$ and set \[n^+(r)=\min\{n\geqslant 2: q_n^+\leqslant r\} \quad \text{and }\quad n^-(r)=\min\{n\geqslant 2: q_n^-\leqslant r\}~.\] Then $q_{n^\pm(r)}^\pm  \leqslant r \leqslant q_{n^\pm(r)-1}^\pm$ and we have \[O(\gamma_{n^+(r)})\subset B_d(x,q_{n^+(r)}^+)\subset B_d(x,r) \subset B_d(x,q_{n^-(r)-1}^-) \subset O(\gamma_{n^-(r)-1})~.\] Thus for all $r$ small enough, we have \[\frac{-\alpha_{n^-(r)-1} + \log C}{\log q_{n^-(r)}^-}\leqslant \frac{\log \nu(B_d(x,r))}{\log r} \leqslant \frac{-\alpha_{n^+(r)} - \log C}{\log q_{n^+(r)-1}^+}~.\]
    Also, notice that \begin{align*}
         & \lim_{n \to \infty} \frac{\log q_n^+}{\log q_{n-1}^+} = \lim_{n \to \infty} \frac{\beta_n-s}{\beta_{n-1}-s} 
         =  \lim_{n \to \infty} \left(\frac{\beta_n - s}{n}\right) \left(\frac{\beta_{n-1} - s}{n-1}\right)^{-1}  = 1~.
    \end{align*}
    Then condition (6) yields
    \begin{align*}
        & \limsup_{r \to 0^+}\frac{-\alpha_{n^+(r)} - \log C}{\log q_{n^+(r)-1}^+} \leqslant \limsup_{n \to \infty} \frac{-\alpha_{n} - \log C}{\log q_{n-1}^+} = \limsup_{n \to \infty} \frac{-\alpha_{n} - \log C}{\log q_{n}^+} \\
         &= \limsup_{n \to \infty} \frac{-\alpha_{n} - \log C}{-(\beta_n-s)} = \limsup_{n \to \infty} \frac{\alpha_{n}/n}{\beta_n/n}=\frac{\lambda_\al}{\lam_\beta}~.
    \end{align*}
    Similarly, we deduce that 
    \begin{align*}
        \liminf_{r \to 0^+}\frac{-\alpha_{n^-(r)-1} + \log C}{\log q_{n^-(r)}^-} & \geqslant \liminf_{n \to \infty} \frac{-\alpha_{n-1} + \log C}{\log q_{n-1}^-}  = \liminf_{n \to \infty} 
        \frac{-\alpha_{n-1}+\log C}{-(\beta_{n-1}+s)} = \frac{\lambda_\al}{\lambda_\beta}~.
    \end{align*}
    Combining these facts we conclude that \[\lim_{r\to 0}\frac{\log \nu(B_d(x,r))}{\log r} = \frac{\lambda_\al}{\lam_\beta}~,\]
    as desired.
\end{proof}

\subsection{Exact dimensionality for Gromov models}\label{section: exact dimensionality from Gromov models}
Recall the setting from \Cref{section: preliminaries} and let $\Gamma<\msf G$ be a $\theta$-Anosov group relative to a collection $\Pc$ of subgroups.
Let $X_{\GM}$ be a Groves--Manning cusp space of $(\Gamma,\Pc)$ and fix a basepoint $o\in X_{\GM}$.
Suppose $(Y, \msf d_Y)$ is another Gromov model for $(\Gamma,\Pc)$ that is orbit-quasi-isometric to $X_{\GM}$. That is, for a (any) basepoint $y_0\in Y$, there exist constants $C\geqslant 1$ and $c\geqslant 0$ such that for any $\gamma\in \Gamma$ we have \[C^{-1}\msf{d}_Y(y_0, \gamma y_0)-c \leqslant \msf{d}_{\GM}(o,\gamma o) \leqslant C \msf{d}_Y(y_0, \gamma y_0) + c~.\] Note that we only need the orbit-quasi-isometry condition in the proof, however, one can verify that an orbit-quasi-isometry extends to a quasi-isometry in this case.
Let $\vis{d}_Y$ be the visual quasi-metric on $\partial Y$ as defined in \Cref{section: visual metrics and shadows}. It induces a quasi-metric on $\Lambda_\theta$ via the homeomorphism $f$, which we still denote by $\vis{d}_Y$. Let $f_{\GM}:\partial X_{\GM} \to \Lambda_\theta$ and $f_Y: \partial Y \to \Lambda_\theta$ be the limit maps.

Let $\psi\in \fraka_\theta^\ast$ be a $(\Gamma,\theta)$-proper linear form that is tangent to the $\theta$-growth indicator of $\Gamma$, and let $\nu_\psi$ be the Patterson--Sullivan measure associated to $\psi$ as in \Cref{definition: Patterson--Sullivan measures}. We proceed to prove \Cref{theoremalpha: exact dimensionality}, asserting that the measure $\nu_\psi$ is exact dimensional with respect to $\vis{d}_Y$.

\begin{proof}[Proof of \Cref{theoremalpha: exact dimensionality}]
    Let $M=\Lam_\theta$, $d=\vis{d}_Y$ and $\nu=\nu_\psi$. We prove the theorem by verifying the conditions in Theorem~\ref{theorem: exact dimensionality framework}.

\medskip \noindent \textbf{(1) The dynamical system.} We choose a compact Borel set $K\subset \Omega_\psi$ satisfying $0<\frakm_\psi(K)<\infty$, which exists by \Cref{theorem: BMS finite and strongly mixing}. Let $(\Fc,\frakm,\Phi)=(\Fc_K,\wh\frakm^\Phi,\Phi)$ be the first return dynamical system associated to $(\Omega,\frakm_\psi,\phi)$ and $K$. 

\medskip \noindent \textbf{(2) The projection.} Let $\wh K \subset \wt \Om_\psi$ be a precompact Borel lift of $K$ for the projection $p_\psi:\wt\Om_\psi\ra \Om_\psi$. For $v\in K$, we let $\wh v\in \wh K$ be its unique lift, and set $\pi(v)={\wh v}^+\in \Lam_\theta$ as in \Cref{section: preliminary on fibered dynamical systems}. Condition~(2) then follows by Poincar\'e recurrence and the fact that $\wt\frakm_\psi$ is equivalent to \[\nu_\psi\otimes \wh \nu_\psi \otimes \mathrm{Leb}_\R\] on $\wt\Omega_\psi$.

\medskip \noindent \textbf{(3) The shadow sets.}
    Let $R\geqslant 0$ be a sufficiently large constant to be determined later. For each $\gamma\in \Gamma$ we define \[O(\gamma)=f_{\GM}(O^{\GM}_R(o,\gamma o))\subset \Lambda_\theta~,\]  and we set $\alpha(\gamma)=\psi(\mu_\theta(\gamma))$. By \Cref{lemma: the comparison for relatively Anosov}, there exist $r_1,r_2>0$ such that \[O_{r_1}^\theta(o',\gamma o')\cap \Lambda_\theta \subset O(\gamma)\subset O_{r_2}^\theta(o',\gamma o')\cap \Lambda_\theta~.\] Then by \Cref{lemma: the shadow lemma for relatively Anosov}, there exists $C\geqslant 1$ such that \[C^{-1}e^{-\alpha(\gamma)} \leqslant \nu_\psi(O(\gamma)) \leqslant C e^{-\alpha(\gamma)} \quad \text{ for all } \gamma\in \Gamma~.\] Thus condition~(3) holds.

\medskip \noindent \textbf{(4) The cocycle.} We let $\msf c_n$ be the induced cocycle associated to $(\Omega_\psi,\frakm_\psi,\phi)$ and $\wh K$, as in \Cref{section: first return}. 

\medskip \noindent \textbf{(5) Visual balls comparison.} Recall the notation from \Cref{section: preliminary on fibered dynamical systems}, and consider the precompact set \[Q=\{\sigma(0):\sigma\in \wt \Psi^{-1}(\wh K)\}\subset X_{\GM}~.\] Recall also that $T_n$ refers to $n$-th return time of the dynamical system $(\Fc,\frakm, \Phi)$.

For a point $v\in \Fc$, we choose any geodesic $\sigma_{\wh v}\in \Gc(X_{\GM})$ with $\Psi(\sigma_{\wh v})=\wh v$, and for each $n\geqslant 1$, choose a number $u_n(v)$ so that \[\wt \sft(\sigma_{\wh v},u_n(v))=T_n(v)~,\] where $\wt \sft$ is given by \Cref{theorem: the flow space for relatively Anosov}. Note that $u_n(v)\geqslant 0$ for all $n$ large enough (depending on $v$) by \Cref{theorem: the flow space for relatively Anosov}.

   As $\sigma_{\wh v}(0)\in Q$ and $\sfc_n(v)^{-1}\sigma_{\wh v}(u_n(v))\in Q$, we have \[\msf{d}_{\GM}(\sigma_{\wh v}(u_n(v)),\sfc_n(v)\sigma_{\wh v}(0)) = \msf{d}_{\GM}(\sfc_n(v)^{-1}\sigma_{\wh v}(u_n(v)),\sigma_{\wh v}(0)) \leqslant \diam (Q)~.\] Moreover $\sigma_{\wh v}([0,\infty))$ is a geodesic ray from $\sigma_{\wh v}(0)$ to $\pi(v)$. Therefore, there is a constant $R_0\geqslant 0$ such that $\pi(v) \in O^{\GM}_{R_0}(o,\msf c_n(v)o)$ for any $v\in \Fc$ and $n\geqslant 1$ large enough (depending on $v$). In addition, by \Cref{lemma: shadow and visual ball comparison} and \Cref{lemma: shadows comparison} there exist constants $R_1,R_2,R_3,s_0\geqslant 0$ such that $O^{\GM}_{R_0}(o,\gamma o) \subset O^Y_{R_1}(y_0,\gamma y_0)$ and hence $\pi(v) \in O^Y_{R_1}(y_0,\msf c_n(v)y_0)$, and \[B_{\vis{d}_Y}(f_Y^{-1}(\pi(v)),e^{-\msf d_Y(y_0,\msf c_n(v) y_0)- s_0}) \subset O^Y_{R_1}(y_0,\msf c_n(v) y_0) \subset (f_Y^{-1}\circ f_{\GM})(O_{R_2}^{\GM}(o,\msf c_n(v)o))\]\[  \subset O^Y_{R_3}(y_0,\msf c_n(v) y_0) \subset B_{\vis{d}_Y}(f_Y^{-1}(\pi(v)),e^{-\msf d_Y(y_0,\msf c_n(v) y_0) + s_0})~.\] 
   Then condition~(5) follows by setting $\beta(\gamma) = d_Y(y_0,\gamma y_0)$ for $\gamma\in \Gamma$.

\medskip \noindent \textbf{(6) Existence of asymptotic limits.} We first prove the statement for $\al$. Let $C=C_Q$ be the constant from \Cref{theorem: the flow space for relatively Anosov}~(5), and for $v\in \Fc$ we consider $\sigma_{\wh v}\in \Gc(X_{\GM})$ and the numbers $u_n(v)$ as in the proof of condition~(5).

Since $\sigma_{\wh v}(0), \ \msf c_n(v)^{-1}\sigma_{\wh v}(u_n(v))\in Q$ for all $n$ large enough, by \Cref{theorem: the flow space for relatively Anosov}~(5) we have \begin{equation}\label{eq:T_n}\abs{T_n(v)-\psi(\mu_\theta(\msf c_n(v)))}\leqslant C_Q~.\end{equation}
    In addition, the measure $\frakm_\psi$ is mixing by \Cref{theorem: BMS finite and strongly mixing}, and hence $\frakm=\wh \frakm^\Phi$ is ergodic for the first return map $\Phi$ by \Cref{lemma: mixing implies ergodic first return}. Therefore, Kac's lemma implies that \[\frac{T_n(v)}{n}\to \frac{1}{\frakm_\psi(K)}=\lam_\al \quad \text{as} \quad n\to \infty \quad \text{for $\frakm$-almost every } v\in \Fc~.\]  Hence, from \eqref{eq:T_n} we deduce \[\frac{\alpha(\msf c_n(v))}{n} = \frac{\psi(\mu_\theta(\msf c_n(v)))}{n} \to \lam_\al~\]
    as $n\to \infty$ for $\frakm$-almost every $v\in \Fc$. 

    For the statement for $\beta$, we apply \Cref{theorem: the flow space for relatively Anosov}~(3) to find constants $a,a',B>0$ such that 
    \[a \msf{d}_{\GM}(o, \msf c_n(v) o) -B \leqslant T_n(v) \leqslant a'\msf{d}_{\GM}(o, \msf c_n(v) o) + B~\]
    for all $n$ and $\frakm$-almost every $v\in \Fc$. By assumption, $X_{\GM}$ and $Y$ are orbit-quasi-isometric, so we can find constants $A_0\geqslant 1$, and $B_0\geqslant 0$ such that \begin{equation}\label{eq:quasiisometry} A_0^{-1}T_n(v) - B_0 \leqslant \beta(\msf c_n(v)) = \msf d_Y(y_0, \msf c_n(v)y_0) \leqslant A_0 T_n(v) + B_0~.\end{equation}
    The upper bound of $\beta(\msf c_n(v))$ in \eqref{eq:quasiisometry} then implies that $\beta\circ \msf c_1$ is $\frakm$-integrable. Then by the subadditivity of $\beta\circ \msf c_n$, Kingman's ergodic theorem, and the ergodicity of $\frakm$ from \Cref{lemma: mixing implies ergodic first return}, there exists a constant $\lam_\beta$ such that \[\frac{\beta(\msf c_n(v))}{n}\to \lam_\beta \quad \text{as} \quad n\to \infty \quad \text{for $\frakm$-almost every } v\in \Fc~.\] The lower bound of $\beta(\msf c_n(v))$ in \eqref{eq:quasiisometry} implies that $\lambda_\beta>0$. This concludes the proof of condition~(6) and hence of the theorem.
\end{proof}

\section{Hyperbolicity of the scalar Cartan metrics}\label{section: hyperbolicity of relatively Morse}

In this section we prove \Cref{theoremalpha: relativemorse implies hyperbolic}, asserting that proper linear linear forms on relatively Morse groups induce Gromov hyperbolic metrics. We then apply this result in combination with \Cref{theoremalpha: exact dimensionality} to deduce \Cref{corollaryalpha: premetric exact dimensionality}.

\subsection{Relatively Morse groups}\label{section: relatively Morse groups}

We first recall the following notions from a series of papers by Kapovich--Leeb \cite{kapovich2017discrete, kapovich2018finsler, kapovich2018relativizing} and Kapovich--Leeb--Porti \cite{kapovich2014morse, kapovich2016some, kapovich2017anosov, kapovich2017discrete, kapovich2018morse}.

Let $\tau\in\Fc_\theta$ and $x\in \msf G / \msf K$. The $\theta$-Weyl cone based at $x$ and pointing toward $\tau$, denoted by $V(x,\tau)\subset \msf G/\msf K$, is the union of all full Weyl chambers based at $x$ (seen as subsets of $\msf G/\msf K$ via the exponential map) whose ideal chambers project to $\tau$ via the canonical projection $\Fc \to \Fc_\theta$. Given another point $y\in \msf G / \msf K$, write $x = g \msf K$ and $y=g\exp(H)\msf K$ for some $g\in \msf G$ and $H\in\overline{\fraka^+}$. We say that the pair $(x,y)$ is $\theta$-\emph{regular} if $\alpha(H)>0$ for any $\alpha\in \theta$. If $(x,y)$ is a $\theta$-regular pair, let $\tau_+(x,y) = g \msf P_\theta \in\Fc_\theta$ be the forward flag determined by the oriented segment from $x$ to $y$, and let $\tau_-(y,x)=gw_0P_{i(\theta)}\in\Fc_{i(\theta)}$ be the opposite backward flag. Note that the definitions of $\tau_+(x,y)$ and $\tau_-(y,x)$ are independent of the choice of $g$ and $H$.

\begin{definition}[Diamonds]
    Suppose $(x,y)$ is a $\theta$-regular pair in $\msf G / \msf K$. The $\theta$-\emph{diamond with extremities} $x,y$ is \[\Diam^\theta(x,y) = V(x,\tau_+(x,y)) \cap V(y,\tau_-(y,x))\subset \msf G/\msf K~.\]
\end{definition}

\begin{definition}[Morse quasi-isometric embedding]
    Let $(Y,\msf{d}_Y)$ be a geodesic metric space. A map $F:Y\to \msf G / \msf K$ is called a $\theta$-\emph{Morse quasi-isometric embedding} if it is a quasi-isometric embedding and there exists $M\geqslant 0$ such that for every geodesic segment $[x,y]\subset Y$, one has that $F([x,y])$ is contained in the $M$-neighborhood of $\Diam^\theta(F(x),F(y))$ (with respect to the canonical $\msf G$-invariant Riemannian distance on $\msf G/ \msf K$).
\end{definition}

\begin{definition}[Relatively Morse groups]
    Suppose $\Gamma <\msf G$ and $\Gamma$ is hyperbolic relative to a collection $\Pc$ of subgroups.
    We say $\Gamma$ is $\theta$-\emph{Morse relative to a Gromov model} $Y$ of $(\Gamma,\Pc)$, if there exists a $\Gamma$-equivariant $\theta$-Morse quasi-isometric embedding $F:Y\to \msf G / \msf K$. We say $\Gamma$ is a \emph{relatively $\theta$-Morse group} if there exists such a Gromov model $Y$. 
\end{definition}

\begin{theorem}[{\cite[Theorem~8.3]{kapovich2018relativizing}, \cite[Proposition~4.4]{zhu2022relatively}}]
    Relatively $\theta$-Morse groups are relatively $\theta$-Anosov.
\end{theorem}

\subsection{Hyperbolicity of relatively Morse groups}
We prove \Cref{theoremalpha: relativemorse implies hyperbolic} in this section. Let $\Gamma <\msf G$ be a $\theta$-Morse group relative to $Y$, a Gromov model for $(\Gamma,\Pc)$. We assume that $\theta$ is symmetric, i.e., $\theta = i(\theta)$, and fix $\psi\in \fraka_\theta^\ast$ a $(\Gamma,\theta)$-proper linear form. For any $g \msf K, h \msf K \in \msf G / \msf K$, we define \[D_\psi(g \msf K,h \msf K) = \psi(\mu_\theta(g^{-1}h)),\] which is independent of the choice of $g$ and $h$ as the Cartan
projection is $\msf K$-invariant. We also define $\msf{d}_\psi:Y\times Y \to \Rb$ according to \[\msf{d}_\psi(x,y) = D_\psi(F(x),F(y)).\]

For the proof of \Cref{theoremalpha: relativemorse implies hyperbolic}, we require the following criterion for hyperbolicity, which follows from the existence of quasi-centers in geodesic hyperbolic spaces (see for instance the proof of \cite[Lemma~6.3]{cantrell-reyes.manhattan}).

\begin{lemma}\label{lemma: criterion for hyperbolicity}
    Let $(Y,\msf{d}_Y)$ be a geodesic $\delta$-hyperbolic space, and let $D:Y\times Y\to \Rb$ be a symmetric function such that:
    \begin{enumerate}
        \item $D$ is bounded below by a constant.
        \item (locally bounded) For any $R_0 \geqslant 0$, there exists $R \geqslant 0$ such that whenever $\msf{d}_Y(x,y)\leqslant R_0$, one has $D(x,y)\leqslant R$.
        \item (coarsely additive near geodesics) For any $R_0 \geqslant 0$, there exists $R \geqslant 0$ such that whenever $z$ is in the $R_0$-neighborhood of a geodesic segment between $x$ and $y$, one has \[\abs{D(x,z)+D(z,y)-D(x,y)} \leqslant R~.\]
    \end{enumerate}
    Then there exists $C>0$ such that 
    \begin{equation*}
    \overline D(x,y)=\begin{cases}
        \ D(x, y)+ C & \text{ if } x\neq y\\
        \ 0 & \text{ if } x=y
\end{cases}
    \end{equation*} 
    defines a roughly geodesic, Gromov hyperbolic metric on $Y$. \qed
\end{lemma}

\begin{proof}[Proof of \Cref{theoremalpha: relativemorse implies hyperbolic}]
    We first verify that $D=\msf{d}_\psi$ satisfies the assumptions in \Cref{lemma: criterion for hyperbolicity}. Note that $\msf{d}_\psi$ is symmetric by our assumption on $\theta$ and bounded below by our choice of $\psi$.

    To verify that that $\msf{d}_\psi$ is locally bounded, we let $\msf{d}_{\msf G / \msf K}$ denote the Riemannian distance on $\msf G / \msf K$. Then there exists a constant $C>0$ such that $D_\psi \leqslant C \msf{d}_{\msf G / \msf K}$. We assume that $A\geqslant 1$ and $B\geqslant 0$ are the constants such that $F$ is an $(A,B)$-quasi-isometric embedding. Therefore, if $\msf{d}_Y(x,y)\leqslant R_0$, then
    \begin{align*}
        \msf{d}_\psi(x,y) & = D_\psi(F(x),F(y)) \leqslant C \msf{d}_{\msf G / \msf K}(F(x),F(y)) \\ & \leqslant C(A \msf{d}_Y(x,y) +B)\leqslant C(AR_0+B) = R_1~.
    \end{align*}

    Next we verify that $\msf{d}_\psi$ is coarsely additive near geodesics. Suppose that $z$ lies in the $R_0$-neighborhood of a geodesic segment $[x,y]\subset Y$ and choose $z_0\in[x,y]$ with $\msf{d}_Y(z,z_0)\leqslant R_0$.  Since $F$ is $\theta$-Morse, there exists \[u \msf K\in \Diam_\theta(F(x),F(y))\] such that the Riemannian distance between $u \msf K$ and $F(z_0)$ is at most the Morse constant $M\geqslant 0$. Take $R_2 = M+R_1$, and consider a constant $R_3\geqslant 0$ such that if $g \msf K,h \msf K\in \msf G / \msf K$ satisfy $\msf{d}_{\msf G / \msf K}(g \msf K,h \msf K)\leqslant R_2$, then $\norm{\mu(g) - \mu(h)}\leqslant R_3$ (see for example \cite[Lemma~4.6]{benoist1997proprietes} or \cite[Lemma~2.1]{dey2024ahlfors}). Then \[D_\psi(F(x),u \msf K) - D_\psi(F(x),F(z)) \leqslant CR_3~,\] and similarly,  \[D_\psi(u \msf K,F(y)) - D_\psi(F(z),F(y)) \leqslant CR_3~.\]
    Since the Cartan projection is additive in a Weyl chamber (see for example \cite[Lemma~4.10]{dey2024ahlfors}), we get \[D_\psi(F(x),u \msf K)+ D_\psi(u \msf K,F(y)) = D_\psi(F(x),F(y))~.\] Therefore \[\abs{\msf{d}_\psi(x,z)+\msf{d}_\psi(z,y) - \msf{d}_\psi(x,y)} \leqslant 2CR_3~.\]

    Up to this point, \Cref{lemma: criterion for hyperbolicity} implies that $\ov{\msf{d}}_\psi$ is Gromov hyperbolic and roughly geodesic for some constant $C$. The fact that $(Y,\ov{\msf{d}}_\psi)$ is $\Gamma$-equivariantly quasi-isometric to a Groves--Manning cusped space follows from \cite[Theorem~1.7, Proposition~1.13]{zhu2022relatively}.
\end{proof}

\subsection{Compatibility of Busemann functions}\label{subsection:compatibility Busemann}

We follow the assumptions from the previous section. Recall the $\fraka_\theta$-valued Busemann function $\beta^\theta$ and Gromov product $G^\theta$ from \Cref{section: Lie groups}. We show that, when restricted to $Y$, their evaluation by $\psi$ are precisely the Busemann function and Gromov product of $\ov{\msf{d}}_\psi$ in the classical sense of Gromov hyperbolicity. We rely on the following lemma, which is a direct adaptation of \cite[Lemma~6.6]{quint2002mesures} to our setting.

\begin{lemma}\label{lemma: compatibility of Busemann functions}
    Suppose $\xi \in\Fc_\theta$ and $\xi_0\in \Fc$ is a lift of $\xi$. Consider a sequence $(g_m)_{m\in \Nb}\subset \msf G$ such that
    \begin{itemize}
        \item $\alpha(\mu(g_m)) \to +\infty$ as $m \to \infty$ for all $\al\in \theta$.
        \item $(g_m \msf K)_{m\in \Nb}$ converges to $\xi$ in the $\theta$-flag compactification of $\msf G/\msf K$.
    \end{itemize} Then for every $r\in \msf G$ we have \[\lim_{m\to \infty} (\mu_\theta(rg_m)-\mu_\theta(g_m)) = p_\theta(\sigma(r,\xi_0))~.\] In particular, for any $g,h\in \msf G$, we have \[\beta^\thet_\xi(g,h) = \lim_{m\to \infty} (\mu_\theta(g^{-1} g_m) - \mu_\theta(h^{-1} g_m))\] \[\text{and } \psi(\beta^\thet_\xi(g,h)) =  \lim_{m\to \infty} (D_\psi(g \msf K, g_m \msf K) - D_\psi(h \msf K, g_m \msf K))~.\]
\end{lemma}

\begin{remark}
    The lemma is further generalized in the GPS setting (see \cite[Proposition~3.4, Lemma~4.2]{blayac2024patterson}).
\end{remark}

Recall that for any $\psi\in \fraka_\theta^\ast$ that is symmetric and tangent to the $\theta$-growth indicator of $\Gamma$, Dey--Kim--Oh \cite{dey2024ahlfors} define the premetric on $\Lambda_\theta$ by \[\vis{d}_\psi(\xi,\eta)= \begin{cases}e^{-\psi(G^\theta(\xi,\eta))} & \text{if } \xi\neq \eta~,\\ 0 & \text{if } \xi=\eta~.\end{cases}\] 

The Busemann function associated to $\ov{\msf{d}}_\psi$ in the classical hyperbolic metric sense is defined as \[\beta^Y_\xi(x,y) = \sup_{(x_m)_m}\limsup_{m\to \infty} (\ov{\msf{d}}_\psi(x,x_m)-\ov{\msf{d}}_\psi(y,x_m)) = \sup_{(x_m)_m}\limsup_{m\to \infty} (\msf{d}_\psi(x,x_m)-\msf{d}_\psi(y,x_m)) \] for any $x,y \in Y$  and $\xi \in\Lambda_\theta \cong\partial Y$, where the suprema are taken over all sequences $(x_m)_{m\in \Nb}\subset Y$ that converge to $\xi$. Then \Cref{lemma: compatibility of Busemann functions} implies that for any $x,y \in Y$ and $\xi \in\Lambda_\theta \cong\partial Y$, \[\psi (\beta^\theta_\xi(g,h)) = \beta^Y_\xi(x,y)~,\]
where $g,h\in \msf G$ are such that $g\msf K = F(x)$ and $h \msf K = F(y)$.
It follows that $\psi(G^\theta(\xi,\eta))$ is identified with the Gromov product of $\ov{\msf{d}}_\psi$ in the hyperbolic metric sense and hence $\vis{d}_\psi$ is a visual quasi-metric on $\Lambda_\theta \cong \partial Y$ induced by the hyperbolic metric $\overline{\msf{d}}_\psi$ on $Y$. From this fact together with \Cref{theoremalpha: exact dimensionality}, \Cref{corollaryalpha: premetric exact dimensionality} follows.

\section{Regularity of Manhattan manifolds}\label{section: manhattan manifolds}

In this section we prove \Cref{theorem: manhattan regularity for relatively Anosov representations}, $C^1$-regularity of Manhattan manifolds associated to proper linear forms of relatively Anosov groups. We prove this result in the more general framework of \Cref{theorem: manhattan framework}, which is based in the work of Cantrell--Tanaka \cite{cantrell-tanaka.manhattan}.  
Inspired by the work of Cantrell, Sert and the first-named author \cite{CRS} on hyperbolic groups, we then prove \Cref{corollaryalpha: strict concavity of growth indicator}, strict concavity of the growth indicator for relatively Anosov groups.

\subsection{A framework theorem for regularity}\label{section: a framework theorem for regularity}

In this section we discuss the regularity of Manhattan manifolds. The argument follows the mechanism used by Cantrell--Tanaka \cite[Section~3 and Theorem~1.1]{cantrell-tanaka.manhattan} for hyperbolic groups.

We begin with a general countable group $\Gamma$ with identity element $e$ and left-invariant functions $\msf d_0,\msf d_1,\dots, \msf d_m: \Gamma \times \Gamma \ra \R$. We suppose that these functions are proper in the sense that for any $C>0$ and $j=0,\dots,m$, there are only finitely many $\gamma\in \Gamma$ such that $\msf d_j(e,\gamma)\leqslant C$. We further assume that these functions have finite and positive exponential growth, meaning that for $j=0,\dots,m$ the limit
\[\del_j:=\limsup_{T\to \infty}\frac{1}{T}\log \#\{\gamma\in \Gamma: \msf d_{j}(e,\gamma)\leqslant T\}\]
is positive and finite. Finally, we assume that  $\msf d_0,\msf d_1,\dots, \msf d_m$ are all quasi-isometric via the identity map. That is, there exist constants $A,B\geqslant 1$ such that for all $i,j\in 0,\dots,m$ and $\gam_1,\gam_2\in \Gamma$ we have
\[A^{-1}\msf d_i(\gam_1,\gam_2)-B\leqslant \msf d_j(\gam_1,\gam_2)\leqslant A \msf d_i(\gam_1,\gam_2)+B.\]

\begin{definition}
Let $\Theta:\Rb^m \ra \R$ be the function such that for $\vv=(v_1,\dots,v_m)\in \Rb^m$, $\Theta(\vv)$ is the critical exponent of 
\[s \mapsto \sum_{\gamma\in \Gamma}e^{-\sum_{i=1}^m{v_i \msf d_i(e,\gamma)}-s \msf d_0(e,\gamma)}~.\]
The graph of $\Theta$ is the \emph{Manhattan manifold} of $(\msf d_0,\dots,\msf d_m)$ and we say that $\Theta$ is the \emph{Manhattan manifold parametrization}.
\end{definition}

Our assumptions on $\msf d_0,\dots,\msf d_m$ imply that $\Theta$ is well-defined, and the H\"older inequality implies that it is continuous and convex. Given $\vv\in \R^m$ we define
\[\msf d_{\vv}:=\sum_{i=1}^m{v_i \msf d_i}+\Theta(\vv) \msf d_{0}.\]

The next result establishes a criterion for differentiability of $\Theta$.

\begin{theorem}\label{theorem: manhattan framework}
    Let $\Gamma$ and $\msf d_0,\dots,\msf d_m$ be as above, and let $\Theta$ be the parametrization of the corresponding Manhattan manifold of $(\msf d_0,\dots, \msf d_m)$. Suppose $M$ is a measurable space on which $\Gamma$ acts by measurable isomorphisms. In addition, for each $\vv\in \R^m$ assume that the following data exists.
    \begin{enumerate}
        \item A probability measure $\nu_{\vv}$ on $M$.
        \item For each $\gamma \in \Gamma$, a measurable set $O_\vv(\gamma)\subset M$, and a constant $C_{\vv}\geqslant 1$ such that \[C_{\vv}^{-1}e^{-\msf d_{\vv}(e,\gamma)}\leqslant \nu_{\vv}(O_\vv(\gamma))\leqslant C_{\vv} e^{-\msf d_{\vv}(e,\gamma)} \quad \text{ for all } \gamma\in \Gamma~.\]
        \item A measurable (discrete) dynamical system $(\Fc_\vv,\frakm_\vv,\Phi_\vv)$, where $\frakm_\vv$ is a $\Phi_\vv$-invariant probability measure on $\Fc_\vv$.
        \item A measurable map $\pi_\vv:\Fc_\vv \to M$ such that the image of an $\frakm_{\vv}$-full measure subset of $\Fc_\vv$ is a $\nu_{\vv}$-full measure subset of $M$.
        \item A measurable cocycle \[\msf c_\vv:\Nb\times \Fc_\vv\to \Gamma\] defined on a $\Phi_\vv$-invariant, $\frakm_\vv$-full measure subset of $\Fc_\vv$, satisfying \[\msf c_{\vv,m+n}(v)=\msf c_{\vv,n}(v)\ \msf c_{\vv,m}(\Phi_\vv^n(v))\quad \text{ for } m,n\in \Nb \text{ and }\frakm_\vv\text{-almost all }v\in \Fc_\vv~.\]
        \item For $\frakm_{\vv}$-almost every $v\in \Fc_\vv$ and $n\in \N$ large enough (depending on $v$), we suppose \[\pi_\vv(v)\in O_\vv(\msf c_{\vv,n}(v))~.\]
        \item A constant $\tau_{\vv}^j$ for each $j=1,\dots, m$, such that for $m_{\vv}$-almost every $v\in \Fc_\vv$ we have \[\lim_{n\to \infty} \frac{\msf d_j(e,\msf c_{\vv,n}(v))}{\msf d_0(e,\msf c_{\vv,n}(v))} = \tau_{\vv}^j~.\]
    \end{enumerate}
    If the maps $\vv \mapsto \tau_{\vv}^j$ are continuous on $\R^m$ for $j=1,\dots,m$, then $\Theta$ is continuously differentiable and \[\frac{\partial \Theta}{\partial v_j}(\vv)=-\tau_{\vv}^j \quad \text{ for all } \vv\in \R^m \text{ and }j=1,\dots,m.\]
\end{theorem}

We assume the conditions (1)-(7) from the statement above hold and prove the following lemmas. For a fixed $\vv\in \R^m$ and a subset $A\subset \Gamma$, we set \[S_\vv(A) = \{x\in M : x\in O_\vv(\gamma)\text{ for infinitely many }\gamma\in A\}\subset M~.\]

\begin{lemma}\label{lemma: positive measure accumulation implies divergence}
    Suppose $\vv\in \Rb^m$ and $A\subset \Gamma$ satisfy $\nu_{\vv}(S_\vv(A))>0$. Then $\sum_{\gamma\in A}e^{-\msf d_{\vv}(e,\gamma)}= +\infty$.
\end{lemma}

\begin{proof}
    For every $N\in \Nb$, we consider the set $A_N=\{\gamma\in A : \msf d_{\vv}(e,\gamma)\geqslant N\}$. If there exists $N_0\in \Nb$ such that $A\setminus A_{N_0} = \{\gamma\in A : \msf d_{\vv}(e,\gamma)< N_0\}$ is infinite, then \[\sum_{\gamma\in A}e^{-\msf d_{\vv}(e,\gamma)} \geqslant \sum_{\gamma\in A\setminus A_{N_0}}e^{-N_0} = +\infty~,\] which completes the proof. Therefore, we may assume that $A \setminus A_N$ is finite for each $N\in \Nb$. 
    Then we have the characterization \[S_\vv(A) =\bigcap_{N\in \Nb}\bigcup_{\gamma\in A_N}O_\vv(\gamma)~.\]
    Applying the shadow estimate from condition (2), for all $N\in \Nb$ we obtain 
    \begin{align*}
        \nu_{\vv}(S_\vv(A)) \leqslant \sum_{\gamma\in A_N}\nu_{\vv}(O_\vv(\gamma))
        \leqslant C_{\vv}\sum_{\gamma\in A_N}e^{-\msf d_{\vv}(e,\gamma)}.
    \end{align*}
    From this it follows that if $\sum_{\gamma\in A}e^{-\msf d_{\vv}(e,\gamma)}<\infty$, then \[\nu_{\vv}(S_\vv(A))\leqslant \limsup_{N\to \infty}{\sum_{\gamma\in A_N}e^{-\msf d_{\vv}(e,\gamma)}}=0~,\] contradicting our assumption.
\end{proof}

Next, we fix an index $j$ from $1,2,\dots, m$.
Given $\vv\in \Rb^m$ and $\epsilon>0$, we consider the set \[A_{\vv}(\epsilon)=\left\{\gamma\in \Gamma : \tau_{\vv}^j-\ep \leqslant \frac{\msf d_j(e,\gamma)}{\msf d_0(e,\gamma)}\leqslant \tau_{\vv}^j+\epsilon\right\}~.\] 

\begin{lemma}\label{lemma: manhattan generic divergence}
    For every $\vv\in \Rb^m$ and $\epsilon>0$, we have $\sum_{\gamma\in A_{\vv}(\epsilon)} e^{-\msf d_{\vv}(e,\gamma)}= +\infty$.
\end{lemma}

\begin{proof}
    Let $E_{\vv}\subset \Fc_{\vv}$ be a $\frakm_{\vv}$-full measure subset on which all the assumptions of \Cref{theorem: manhattan framework} hold simultaneously. By condition (7), for $v\in E_{\vv}$ we have \[\lim_{n\to\infty} \frac{\msf d_j(e,\msf c_{\vv,n}(v))}{\msf d_0(e,\msf c_{\vv,n}(v))} = \tau_{\vv}^j~.\] Hence for every $v\in E_{\vv}$, the elements $\msf c_{\vv,n}(v)$ belong to $A_{\vv}(\epsilon)$ for all sufficiently large $n$.  Since $\pi_\vv(v)\in O_\vv(\msf c_{\vv,n}(v))$ for all $n$ large enough by condition (6), it follows that $\pi_\vv(v)\in S_\vv(A_{\vv}(\epsilon))$. But $\pi_\vv(E_{\vv})$ has full $\nu_{\vv}$-measure in $M$ by condition (4), implying $\nu_{\vv}(S_\vv(A_\vv(\ep)))=1$. The conclusion then follows from \Cref{lemma: positive measure accumulation implies divergence}.
\end{proof}

Our next result relates the drifts $\tau_{\vv}^j$ with $\Theta$. We let $\ve_j\in \Rb^m$ denote the vector whose $j$-th coordinate is 1 and the rest are all zero.

\begin{lemma}\label{lemma: manhattan slope inequalities}
    For every $\vv\in \Rb^m$ and $h\in \R$ we have $\Theta(\vv+h\ve_j)\geqslant \Theta(\vv)-h\tau_{\vv}^j$. Consequently, if $h>0$ then we have \[-\tau_{\vv}^j\leqslant \frac{\Theta(\vv+h\ve_j)-\Theta(\vv)}{h}\leqslant -\tau_{\vv+h\ve_j}^j~.\]
\end{lemma}

\begin{proof}
    Suppose $\vv=(v_1,\dots,v_m)$. We first assume $h>0$ and show that for any $\epsilon>0$, \[s_\epsilon=\Theta(\vv)-h(\tau_{\vv}^j+2\epsilon) \leqslant \Theta(\vv+h\ve_j)~.\] If $\gamma\in A_{\vv}(\epsilon)$, then $\msf d_j(e,\gamma)\leqslant (\tau_{\vv}^j+\epsilon)\msf d_0(e,\gamma)$, and hence
    \begin{align*}
        &-\sum_{i=1}^m{v_i \msf d_i(e,\gamma)}-h \msf d_j(e,\gamma)-s_\epsilon  \msf d_0(e,\gamma)\\
        = & -\sum_{i=1}^m{v_i \msf d_i(e,\gamma)}-\Theta(\vv) \msf d_0(e,\gamma)+h((\tau_{\vv}^j+2\epsilon) \msf d_0(e,\gamma)- \msf d_j(e,\gamma)) \geqslant -\msf d_{\vv}(e,\gamma)~.
    \end{align*}
    Consequently, we have \[\sum_{\gamma\in A_{\vv}(\epsilon)}e^{-\sum_{i=1}^m{v_i \msf d_i(e,\gamma)}-h  \msf d_j(e,\gamma)-s_\epsilon  \msf d_0(e,\gamma)} \geqslant \sum_{\gamma\in A_{\vv}(\epsilon)}e^{-\msf d_{\vv}(e,\gamma)}~.\] 
    By \Cref{lemma: manhattan generic divergence}, the series on the right hand side diverges, so does the series on the left hand side, implying $s_\epsilon\leqslant \Theta(\vv+h\ve_j)$. Since $\epsilon>0$ is arbitrary, we get \[\Theta(\vv+h\ve_j)\geqslant \Theta(\vv)-h\tau_{\vv}^j~.\]
    The case $h<0$ is analogous.
\end{proof}

\begin{proof}[Proof of \Cref{theorem: manhattan framework}]
    The function $\Theta$ is continuous and convex, so it has left and right partial derivatives at every point and is differentiable almost everywhere. Then by \Cref{lemma: manhattan slope inequalities}, we deduce that $\Theta$ is differentiable with partial derivatives $\frac{\partial\Theta}{\partial v_j}(\vv)=-\tau_{\vv}^j$ for all $\vv$ and $j=1,\dots, m$. Then $\Theta$ is $C^1$-regular whenever $\tau_{\vv}^j$ are continuous on $\R^m$ for $j=1,\dots,m$.
\end{proof}

\subsection{$C^1$-regularity for relatively Anosov groups}

For the rest of the section we suppose that $\Gamma<\msf G$ is a $\theta$-Anosov group relative to a collection $\Pc$ of subgroups. Any linear form $\psi\in \fraka_\theta^\ast$ induces a left-invariant function $\msf{d}_\psi:\Gamma \times \Gamma\ra \R$ according to
\[\msf{d}_\psi(\gam_1,\gam_2)=\psi(\mu_\theta(\gam_1^{-1}\gam_2)) \quad \text{ for all }\gam_1,\gam_2\in \Gamma.\]
We fix $(\Gamma,\theta)$-proper linear forms $\psi_0,\psi_1,\dots,\p_m \in \fraka_\theta^\ast$ and let $\Theta:\R^m \ra \R$ be the parametrization of the Manhattan manifold associated to $(\msf d_{\p_0},\dots,\msf d_{\p_m})$. 

Given $\vv\in \R^m$, we let
\begin{equation}\label{eq:p_vv}\p_{\vv}:=\sum_{i=1}^m{v_i\p_i}+\Theta(\vv)\p_{0}\in \fraka_\theta^\ast
\end{equation}
and abbreviate $d_\vv=d_{\p_\vv}$. Then each $\p_\vv$ is $(\Gamma,\theta)$-proper and tangent to the $\theta$-growth indicator of $\Gamma$, i.e. it satisfies $\del_{\p_{\vv}}(\Gamma)=1$.

By verifying the assumptions in \Cref{theorem: manhattan framework}, we prove the following, which confirms \Cref{theoremalpha: regularity of Manhattan manifold} from the introduction.

\begin{theorem}\label{theorem: manhattan regularity for relatively Anosov representations}
    The map $\Theta:\Rb^m \to \R$ is $C^1$-regular.
\end{theorem}

Our measurable space is $M=\Lam_\theta$, and for $\vv\in \R^m$ we let $\nu_{\vv}=\nu_{\p_\vv}$ be the unique $(\Gamma,\p_{\vv})$-Patterson--Sullivan probability measure on $\Lam_\theta$, and $\wh \nu_{\vv}$ be the unique $(\Gamma,\p_{\vv} \circ i)$-Patterson--Sullivan probability measure on $\Lam_{i(\theta)}$, as in \Cref{definition: Patterson--Sullivan measures}. This settles condition (1).

For condition (2), we fix a basepoint $o'\in \msf G / \msf K$ and define the corresponding shadow at $\gamma\in \Gamma$ by $$O_\vv(\gamma) =O^\thet_{R_\vv}(o',\gamma o')\cap \Lam_\theta.$$ We choose $R_\vv>0$ sufficiently large so that \Cref{lemma: the shadow lemma for relatively Anosov} implies \[C_{\vv}^{-1} e^{-\p_{\vv}(\mu_\theta(\gamma))} \leqslant \nu_{\vv}(O^\theta_{R_\vv}(o',\gamma o')) \leqslant C_{\vv} e^{-\p_{\vv}(\mu_\theta(\gamma))}~\] 
for a constant $C_\vv$ independent of $\gamma$. 

To define the dynamical systems from condition (3) and check conditions (4) and (5), we fix the flow space $(\wt \Omega,\phi)$ with \[\wt\Omega=\Lambda_\theta^{(2)}\times\R, \quad \text{ and } \quad \phi^s(\xi,\eta,r)=(\xi,\eta,r+s)~.\]
We also fix the action of $\Gamma$ on $\wt\Omega$ induced by $\p_0$ as in \Cref{section: preliminary on fibered dynamical systems}, with quotient $\Omega_0=\Omega_{\p_0}$ and quotient map $p:\wt\Omega \ra \Omega_0$. Note that each $\p_{\vv}$ induces a different action on $\wt\Omega$, which we call the $(\Gamma,\p_{\vv})$-\emph{action} if we want to emphasize the dependence on $\p_{\vv}$. The quotient for this action is denoted by $\Omega_{\vv}=\Omega_{\p_{\vv}}$. The flow on $\wt\Omega$ induces a flow on each $\Omega_{\vv}$, which we again denote by $\phi$. 
\\

We fix $\vv\in \Rb^m$, for which we want to find an appropriate compact set $\wh K_\vv\subset \wt \Omega$ to define the first return map, as in \Cref{section: first return}. Let $\sfL_{\vv}=\sfL_{\p_{\vv}}$ be the Bowen--Margulis current on $\Lam_\theta^{(2)}$, and $\wt\frakm_{\vv}=\wt\frakm_{\p_{\vv}}=\sfL_{\vv} \otimes \mathrm{Leb}_\R$ be the measure on $\wt\Omega$ as in \Cref{section: Patterson--Sullivan measures}. Note that each measure $\wt\frakm_{\vv}$ commutes with the $(\Gamma,\p_0)$-action, so it descends to a Radon measure $\frakm_{\vv}^0$ on $\Omega_0$. Be aware that this measure is different from the Bowen--Margulis--Sullivan measure $\frakm_{\p_{\vv}}$ on $\Omega_{\vv}$.

Since $\wt\Omega$ is locally compact metrizable, the flow $\phi$ is continuous, the $(\Gamma,\p_0)$-action on $\wt\Omega$ is properly discontinuous, and each measure $\wt\frakm_{\vw}^0$ on $\Omega_0$ is Radon and has full support, from \Cref{lemma: good local section} we immediately obtain the following.

\begin{lemma}\label{lemma: choice of K}
    There exists a compact set $\wh K_\vv\subset \wt \Omega$ satisfying:
        \begin{itemize}
        \item $p$ restricts to an embedding of $\wh K_\vv$ into $\Omega_0$.
        \item $\wt \frakm_{\vw}(\wh K_\vv)>0$ for all $\vw\in \R^m$.
        \item $\wt \frakm_{\vv}(\partial \wh K_\vv)=0$.
    \end{itemize}
\end{lemma}

In what follows, we fix a set $\wh K_\vv$ as in the above lemma and define $K_\vv=p(\wh K_\vv)\subset \Omega_0$. To apply the discussion from \Cref{section: first return} we require the following lemma.

\begin{lemma}\label{lemma: finiteness and ergodicity}
    For each $\vv\in \Rb^m$, the measure $\frakm^0_{\vv}$ on $\Omega_0$ is finite and ergodic.
\end{lemma}

For the proof of this lemma and subsequent results, we let $(X_{\GM},\msf{d}_{\GM})$ be a Groves--Manning cusp space for the pair $(\Gamma,\Pc)$ with induced limit map $f:\partial X_{\GM} \to \Lambda_\theta$ and geodesic flow space $(\Gc(X_{\GM}),\varphi)$, as in \Cref{section: relatively hyperbolic groups}. Also, for $j=0,1,\dots,m$, we apply \Cref{theorem: the flow space for relatively Anosov} to $\psi_j$ and let \[\wt\Psi_j:\Gc(X_{\GM})\to \wt \Omega \quad \text{and} \quad \wt \sft_j:\Gc(X_{\GM})\times\Rb \to \Rb\] be the associated reparametrization and translation cocycle, respectively. We also let $\wh \sft_j:\wt\Omega \times \R \ra \R$ and $\ov\sft_j:\Omega_{\p_j} \times \R \ra \R$ be the corresponding inverse almost cocycles.

\begin{proof}[Proof of \Cref{lemma: finiteness and ergodicity}]
  We first provide a sketch for the proof of finiteness of $\frakm_{\vv}$ as it closely follows the proof of \cite[Theorem~9.1]{kim2025relatively}. There is a $\Gamma$-invariant decomposition \[\wt \Omega = \wt \Psi_0(\Gc_{\mathrm{thick}})\cup \Gamma \left(\bigcup_{P\in \Pc}\Gc_P\right)~\]
  so that the action of $\Gamma$ on $\wt\Psi_0(\Gc_{\mathrm{thick}})$ is cocompact. Then $\frakm_{\vv}^0(\Gamma \bs \wt\Psi_0(\Gc_{\mathrm{thick}}))$ is finite since $\frakm_{\vv}^0$ is Radon. In addition, for each $P\in \Pc$ and any $\ep>0$ small enough there is $C>0$ such that
    \begin{align*}
        \frakm_{\vv}^0(\Gc_P) \leqslant & \sum_{\gamma\in P} (C+\psi_0(\mu_\theta(\gamma))) e^{-(1-\epsilon)\psi_{\vv}(\mu_\theta(\gamma))} \\
        \leqslant & \sum_{\gamma\in P} (C+C\psi_{\vv}(\mu_\theta(\gamma))+C) e^{-(1-\epsilon)\psi_{\vv}(\mu_\theta(\gamma))} < +\infty.
    \end{align*}
    Here, the first inequality follows as in the proof of the bound for $m_\psi(P \bs \wt \Psi(\Gc_P))$ in \cite[Theorem~9.1]{kim2025properly}, the second inequality follows from \Cref{lemma: comparison positive linear forms}, and the third inequality follows from \cite[Corollary~7.2]{canary2025patterson} and the fact that $\del_{\p_{\vv}}(\Gamma)=1$.

    Next, we show that $\frakm^0_{\vv}$ is ergodic for the flow $\phi$ on $\Omega_0$. Indeed, by \cite[Corollary~1.5]{canary2025patterson} it follows that the measure $\sfL_{\vv}$ on $\Lambda_\theta^{(2)}$ is ergodic for the diagonal action of $\Gamma$. Then a standard Hopf argument implies that $\frakm^0_{\vv}$ is ergodic (see for example \cite[Theorem~11.2, Section~11.5]{blayac2024patterson}).
\end{proof}

By \Cref{prop:m_tctnous+finite}, we let $(\Fc_{\vv},\wh \frakm_{\vv},\Phi_{\vv}) = (\Fc_{K_\vv}, \wh \frakm_{\vv}^{\Phi_\vv}, \Phi_\vv)$ be the first return dynamical system induced by $(\Omega_0, \frakm_{\vv}^0, \phi)$ and $K_\vv$, and let $(\msf c_{\vv,n})_n$ be the cocycle induced by this first return and $\wh K_\vv$. We also set $\pi_{\vv}: \Fc_\vv\to \Lambda_\theta$ according to $\pi(v)= v^+$. This is the data from conditions (3)-(5), and the full measure assertion from condition (5) follows from the local product structure of $\wt\frakm_\vv$. 

Condition (6) holds by the proof of condition (5) in \Cref{theoremalpha: exact dimensionality}, so we are left to prove condition (7). The values $\tau_\vv^j$ in that condition are given by the next lemma.

\begin{lemma}\label{lem:deftau}
    For each $\vv\in \Rb^m$ there exists a positive number $\tau_{\vv}^j$ such that for $\wt\frakm_{\vv}$-almost every $v\in \wt\Omega$ we have
    \[\lim_{ s\to \infty} \frac{\wh \sft_0(v,s)}{\wh \sft_j(v,s)} = \tau_{\vv}^j~.\]
\end{lemma}

\begin{proof}
    Consider the almost cocycle $\ov \sft_0$ on $\Omega_0$. By \Cref{theorem: the flow space for relatively Anosov} and the definition of $\wh \sft_0$, we can find constants $a,a'>0$ such that for all $v\in \Omega_0$ and all $s>0$ large enough we have 
    \begin{equation}\label{eq:ineqt_0}
        a\leqslant \frac{\ov\sft_0(v,s)}{s}\leqslant a'.
    \end{equation}
    On the other hand, the measure $\frakm^0_{\vv}$ is finite and ergodic by \Cref{lemma: finiteness and ergodicity}, and hence $\ov \sft_0$ is a $\frakm_{\vv}^0$-integrable almost cocycle by \Cref{lemma: hat t almost cocycle}. By Kingman's subadditive ergodic theorem and \Cref{eq:ineqt_0}, there exists a positive number $\lam^0_{\vv}>0$ such that for $\wt\frakm_{\vv}$-almost every $v\in \wt\Omega$ we have \[\lim_{ s\to \infty} \frac{\wh \sft_0(v,s)}{s} = \lam^0_{\vv}~.\]
    The exactly same argument applied to $\ov\sft_j$ on $\Omega_{\p_j}$ and the measure induced by $\wt\frakm_{\vv}$ on this set gives us the corresponding positive number $\lam_{\vv}^j$. The result then follows by setting $\tau^j_{\vv}=\lam_{\vv}^0/\lam_{\vv}^j$.
\end{proof}

To complete the proof of \Cref{theorem: manhattan regularity for relatively Anosov representations}, it remains to show the continuity of $\tau_{\vv}^j$ and the limit identity in condition (7). This is done in the next section. 

\subsection{Continuity of the drifts}
We keep the setting and notation as in previous section. We fix $\vv^0\in \R^m$ and abbreviate $\wh K=\wh K_{\vv^0}, K=K_{\vv^0}, \Phi=\Phi_{\vv^0}$, and $\msf c=\msf c_{\vv^0}$. Our goal is to prove the following result.

\begin{proposition}\label{prop:taucontinuousatt0}
    For any $\vv\in \Rb^m$ we have that \begin{equation}\label{eq:deftau_t}
    \frac{\p_j(\mu_\theta(\msf c_n(v)))}{\p_0(\mu_\theta(\msf c_n(v)))}\to \tau_{\vv}^j \quad \text{ as }n\to \infty
    \end{equation}
    for $\frakm^0_{\vv}$-almost every $v\in K$. Moreover, the function $\vv\mapsto \tau_{\vv}^j$ is continuous at $\vv^0$. 
\end{proposition}

Since $\vv^0$ is arbitrary, we immediately obtain the following, which is the last step in the proof of \Cref{theorem: manhattan regularity for relatively Anosov representations}.
\begin{corollary}\label{coro:tautcontinuous}
    The function $\vv\mapsto \tau_{\vv}^j$ is continuous.
\end{corollary}

We begin the proof of \Cref{prop:taucontinuousatt0} by showing continuity of the measures $\frakm_\vv^0$ as $\vv$ varies. Recall that the \emph{vague topology} on the space of Borel measures on $\Omega_0$ is such that a sequence of measures $\frakn_n$ converges to $\frakn_\infty$ as $n\to \infty$ if and only if $\int{F}{\de \frakn_n}\to \int{F}{\de \frakn_\infty}$ for each compactly supported continuous function $F:\Omega_0\ra \R$. See for instance \cite[Section~13.2]{klenke}.

\begin{proposition}\label{prop:m_tctnous+finite}
   The assignment $\vv\mapsto \frakm_{\vv}^0$ is continuous in the vague topology.
\end{proposition}

The proof of this proposition is split into two lemmas. 

\begin{lemma}\label{lemma: weak star continuity PS densities}
    The maps $\vv\mapsto \nu_{\vv}$ and  $\vv\mapsto \wh\nu_{\vv}$ are weak* continuous.
\end{lemma}

\begin{proof}
    We only prove the statement for $\nu_{\vv}$ as the proof for $\wh\nu_{\vv}$ is identical. We fix $\vv\in \Rb^m$ and let $(\vv_n)_{n\in \Nb}$ be a sequence in $\Rb^m$ that converges to $\vv$. Since $\Lambda_\theta$ is compact, it is enough to show that any convergent subsequence of $(\nu_{\vv_n})_n$ has $\nu_{\vv}$ as limit. Consider any such subsequence (that we also denote $(\nu_{\vv_n})_n$) and let $\nu$ be its limit. By the uniqueness of $(\Gamma,\p_{\vv})$-conformal measures from \Cref{theorem: PS existence uniqueness ergodicity}, it is enough to show that $\nu$ is $(\Gamma,\p_{\vv})$-conformal. 

    Since each $\nu_{\vv_n}$ is $(\Gamma,\p_{\vv_n})$-conformal, given any $\gamma\in\Gamma$ and any continuous function $F\in C(\Lambda_\theta)$ we have \[\int_{\Lambda_\theta} F(\gamma\xi)\, \de \nu_{\vv_n}(\xi) = \int_{\Lambda_\theta} F(\xi) e^{\psi_{\vv_n}(\beta^\theta_\xi(o',\gamma o'))}\, \de\nu_{\vv_n}(\xi)~.\] In addition, the Busemann function $\xi \mapsto \beta^\theta_\xi(o',\gamma o')$ is continuous, and  $\psi_{\vv_n}\to\psi_{\vv}$ in $\mathfrak a_\theta^\ast$ because $\Theta$ is continuous. Therefore, the weights $e^{\psi_{\vv_n}(\beta^\theta_\xi(o',\gamma o'))}$ converge to $e^{\p_{\vv}(\beta_\xi^\theta(o',\gamma o'))}$ uniformly on $\xi\in \Lambda_\theta$. By passing to the limit we obtain \[\int_{\Lambda_\theta} F(\gamma\xi)\, \de \nu(\xi) = \int_{\Lambda_\theta} F(\xi) e^{\psi_{\vv}(\beta^\theta_\xi(o',\gamma o'))}\, \de \nu(\xi)~,\]
    and hence $\nu$ is $(\Gamma,\p_{\vv})$-conformal.
\end{proof}

The next lemma completes the proof of \Cref{prop:m_tctnous+finite}.

\begin{lemma}\label{lem:continuousm_t}
    The map $\vv\mapsto \sfL_{\vv}$ is continuous for the vague topology on Borel measures on $\Lambda_\theta^{(2)}$. Consequently, the measures $\frakm_{\vv}^0$ on $\Omega_0$ vary continuously for the vague topology.
\end{lemma}

\begin{proof}
    Let $\vw\in \Rb^m$ and $F$ be a compactly supported continuous function on $\Lambda_\theta^{(2)}$. Since $G^\theta$ is continuous on $\Lambda_\theta^{(2)}$ and $\p_{\vv}$ converges to $\p_{\vw}$ in $\fraka_\theta^\ast$, we have that \[F(\xi,\eta)e^{2\psi_{\vv}(G^\theta(\xi,\eta))}\to F(\xi,\eta)e^{2\psi_{\vw}(G^\theta(\xi,\eta))}\] uniformly as $\vv\to \vw$. By \Cref{lemma: weak star continuity PS densities} and \cite[Theorem~2.8]{billingsley2013convergence}, $\nu_{\vv}\otimes\wh\nu_{\vv}$ weak* converges to $\nu_{\vw}\otimes\wh\nu_{\vw}$. Therefore, \[\int_{\Lambda_\theta^{(2)}} F\, \de\sfL_{\vv}\longrightarrow \int_{\Lambda_\theta^{(2)}} F\, \de \sfL_{\vw}~\]
    and $\sfL_{\vv}$ vaguely converges to $\sfL_{\vw}$. This readily implies that $\wt\frakm_{\vv}=\sfL_{\vv} \otimes \mathrm{Leb}_\R$ vaguely converges to $\wt\frakm_{\vw}$, which in turn implies that $\frakm_{\vv}^0$ vaguely converges to $\frakm_{\vw}^0$.
\end{proof}

For $v\in K$ and $n\geqslant 1$, recall that $T_n(v)\in \N$ is the time such that $\Phi^n(v)=\phi^{T_n(v)}(v)$, which is defined for $\frakm_\vv^0$-almost every $v\in K$ for each $\vv$.

\begin{lemma}\label{lem:almostadditive}
    For $j=0,1,\dots,m$, there exists a constant $B_j\geqslant 0$, depending only on $\wh K$, such that for all $n_1,n_2\geqslant 1$ and every $\vv\in \Rb^m$, for $\frakm_{\vv}^0$-almost every $v\in K$ we have  \[\abs{\p_j(\mu_\theta(\msf c_{n_1+n_2}(v)))-\p_j(\mu_\theta(\msf c_{n_1}(\Phi^{n_2}(v))))-\p_j(\mu_\theta(\msf c_{n_2}(v)))} \leqslant B_j~.\]
    Moreover, for all $n\geqslant 1$ and $\frakm^0_{\vv}$-almost every $v\in K$ we have \[\abs{\p_0(\mu_\theta(\msf c_n(v)))-T_n(v)}\leqslant B_0~.\]
\end{lemma}

\begin{proof}
    Fix $j=0,1,\dots,m$ and consider the set \[\Qc = \{\sigma(0)\in X_{\GM}:\wt\Psi_0(\sigma)\in \wh K\}\subset X_{\GM}~,\]
    which is compact by \Cref{theorem: the flow space for relatively Anosov}~(4) and the properness of $X_{\GM}$. Let $\Bc$ be the set of triples $(\sig,s,\gamma)$ with  $\sig\in \Gc(X_{\GM}), s\geqslant 0$ and $\gamma\in \Gamma$ satisfying $\sigma(0)\in \Qc$ and $\gamma^{-1}\sigma(s)\in \Qc$.
    Then \Cref{theorem: the flow space for relatively Anosov}~(5) gives us a constant $C_j\geqslant 0$ such that \begin{equation}\label{eq:compt_0-psi}
        \abs{\wt \sft_j(\sigma,s) -\psi_j(\mu_\theta(\gamma))}\leqslant C_j \quad \text{ for every } (\sig,s,\gamma)\in \Bc.
    \end{equation}

    Let $v\in K$ be a point such that $\Phi^n(v)$ is well-defined for each $n$, and let $\wh v\in \wh K$ be the lift of $v$. We fix a geodesic $\sigma\in \Gc(X_{\GM})$ such that $\wt\Psi_0(\sigma)=\wh v$ and construct times $0< r_1< r_2 < \cdots $ such that 
    \begin{equation}\label{eq:goodtimes}
        \wt \sft_0(\sig,r_k)=T_k(v) \quad \text{ for all } k\geqslant 1.
    \end{equation}
    
    Since $s\mapsto\widetilde{\mathsf t}_0(\sigma,s)$ is continuous, $\widetilde{\mathsf t}_0(\sigma,0)=0$, and $\widetilde{\mathsf t}_0(\sigma,s)\to+\infty$ as $s\to+\infty$ by \Cref{theorem: the flow space for relatively Anosov}~(3), the numbers \[r_k = \inf\{s\ge0:\widetilde{\mathsf t}_0(\sigma,s)=T_k(v)\}\]
    are well-defined. Moreover, since  $T_k(v)<T_{k+1}(v)$ for each $k$, every path from level $0$ to level $T_{k+1}(v)$ crosses level $T_k(v)$ first, and hence the intermediate value theorem gives us $r_k< r_{k+1}$.
    
    \Cref{eq:goodtimes} and \Cref{theorem: the flow space for relatively Anosov}~(1) also imply that    
    \[\wt\Psi_0(\varphi^{r_k}(\sig))=\phi^{T_k(v)}(\wh v)\in \msf c_k(v)\wh K \quad \text{ for all } k\geqslant 1~.\]
    Combining this with the cocycle property of $\msf c$, for $n_1,n_2\geqslant 1$ we have that the triplets
    \[(\sig,r_{n_2},\msf c_{n_2}(v)), \ \ (\sig,r_{n_1+n_2},\msf c_{n_1+n_2}(v)), \ \ (\msf c_{n_2}(v)^{-1}\varphi^{r_{n_2}}(\sig),r_{n_1+n_2}-r_{n_2},\msf c_{n_1}(\Phi^{n_2}(v)))\]
    all belong to $\Bc$. Therefore, the cocycle property for $\wt \sft_0$ combined with \Cref{eq:compt_0-psi} yield
\[\abs{\p_j(\mu_\theta(\msf c_{n_1+n_2}(v)))-\p_j(\mu_\theta(\msf c_{n_1}(\Phi^{n_2}(v))))-\p_j(\mu_\theta(\msf c_{n_2}(v)))} \leqslant 3C_j=B_j~,\]
as desired. The moreover statement then follows from \Cref{eq:compt_0-psi} and \Cref{eq:goodtimes}.
\end{proof}

In view of the above lemma, Kac's lemma, and \Cref{lemma: comparison positive linear forms}, for each $j=0,\dots,m$ the functions $v \mapsto \p_j(\mu_\theta(\msf c_n(v)))$ form a $\frakm^0_{\vv}$-integrable almost cocycle for all $\vv\in \Rb^m$. Therefore, Kingman's subadditive ergodic theorem applies, and for $j=0,\dots,m$ and $\vv\in \Rb^m$, the limit
\[\ell^j_{\vv}=\lim_{n\to \infty}\frac{1}{n}\int_K{\psi_j(\mu_\theta(\msf c_n(v)))} \de \frakm^0_{\vv}\]
is well-defined and positive. 

\begin{proof}[Proof of \Cref{prop:taucontinuousatt0}]
First, we prove that each function $\vv\mapsto \ell_{\vv}^j$ is continuous at $\vv^0$, similar to the proof of  \cite[Theorem~6.1]{CMGR}.

By \Cref{lem:almostadditive} and $\Phi$-invariance of $\frakm_\vv^0|_{K}$, the sequence \[a_n=\int_K \psi_j(\mu_\theta(\msf c_n(v))) \de \frakm^0_{\vv}\]
satisfies
\[|a_{m+n}-a_m-a_n|\leqslant C_j\frakm_\vv^0(K) \quad \text{ for }m,n\in \N.~\]
Hence, after applying Fekete's lemma to the subadditive sequences $n\mapsto a_{n}+C_j\frakm_\vv^0(K)$ and $n\mapsto -a_n+C_j\frakm_\vv^0(K)$ we have
\begin{equation}\label{eq.L_B}
        \abs{\ell_{\vv}^j-\frac{1}{n}\int_K \psi_j(\mu_\theta(\msf c_n(v))) \de \frakm^0_{\vv}}\leqslant \frac{B_j\frakm^0_{\vv}(K)}{n}.
    \end{equation}
Then \Cref{prop:m_tctnous+finite} and \Cref{lemma: choice of K} imply that $\vv\mapsto \frakm^0_{\vv}(K)$ is continuous, and hence locally bounded, at $\vv^0$. Therefore, the continuity of $\vv \mapsto \ell_{\vv}^j$ at $\vv^0$ follows from the continuity of $\p_j \circ \mu_\theta$, \Cref{lemma: continuous almost everywhere}, and \Cref{prop:m_tctnous+finite}.

To end the proof of the proposition, for $\vv\in \Rb^m$ we claim that $\tau_{\vv}^j$ satisfies \Cref{eq:deftau_t} for $\frakm_{\vv}^0$-almost every $v\in K$. Indeed, by \Cref{lemma: finiteness and ergodicity} and \Cref{lemma: mixing implies ergodic first return}, the first return map $(K,\Phi)$ is $\frakm_\vv^0$-ergodic. Thus $\frakm_{\vv}^0$-almost every $v\in K$ satisfies
\begin{equation}\label{eq:L_t(v)}
\frac{\psi_0(\mu_\theta(\msf c_n(v)))}{n}\to \ell_{\vv}^0 \quad \text{ and } \quad \frac{\psi_j(\mu_\theta(\msf c_n(v)))}{n}\to \ell_{\vv}^j \quad \text{ as }n\to \infty.
\end{equation}
On the other hand, from the proof of \Cref{lem:almostadditive} we can find a constant $C>0$ such that for all $n\geqslant 1$ and $\frakm_{\vv}^0$-almost every $v\in K$ with lift $\wh v\in \wh K$ and $\sig\in \Gc(X_{\GM})$ with $\wt\Psi_0(\sig)=\wh v$, there exists some $r_n$ (depending on $\sig$) such that
\begin{equation*}
    \abs{\wt \sft_0(\sig,r_n)-\p_0(\mu_\theta(\msf c_n(v)))}\leqslant C \quad \text{ and }\quad \abs{\wt \sft_j(\sig,r_n)-\p_j(\mu_\theta(\msf c_n(v)))}\leqslant C.
\end{equation*}
By \Cref{theorem: the flow space for relatively Anosov}~(3), we have that $r_n\to \infty$ as $n\to \infty$, implying that 
\[\frac{\ell_{\vv}^j}{\ell_{\vv}^0}=\lim_{n\to \infty}{\frac{\p_j(\mu_\theta(\msf c_n(v)))}{\p_0(\mu_\theta(\msf c_n(v)))}}=\lim_{n\to \infty}\frac{\wt\sft_j(\sig,r_n)/r_n}{\wt\sft_0(\sig,r_n)/r_n}~.\]
Finally, by the definition of the inverse almost cocycles $\wh \sft_0$ and $\wh\sft_j$ and \Cref{lem:deftau}, we can find sequences $(u_n)_n$ and $(v_n)_n$ converging to infinity as $n\to \infty$, and such that 
\[\lim_{n\to \infty}\frac{\wt\sft_j(\sig,r_n)/r_n}{\wt\sft_0(\sig,r_n)/r_n}=\lim_{n\to \infty}\frac{\wh\sft_0(\wh v,u_n)/u_n}{\wh\sft_j(\wh v,v_n)/v_n}=\tau_\vv~.\]
We conclude that $\tau_{\vv}^j=\ell^j_{\vv}/\ell_{\vv}^0$ for $\frakm_{\vv}^0$-almost every $v\in K$, which by \Cref{eq:L_t(v)} completes the proof of the proposition.
\end{proof}

\subsection{Strict concavity of growth indicator}\label{subsection: growth} We are ready to prove \Cref{corollaryalpha: strict concavity of growth indicator}. Let $\Gamma <\msf G$ be a Zariski dense relatively $\theta$-Anosov group, and consider the set \[\partial \Dc_\Gamma^\theta=\{\psi\in \fraka_\theta^\ast:\psi \text{ is tangent to }\p_\Gamma^\theta\}=\{\psi\in \fraka_\theta^\ast: \del_\psi(\Gamma)=1\}\subset \fraka_\theta^\ast~.\]

\begin{lemma}\label{lem:ManhattancoundaryC^1}
    $\partial \Dc_\Gamma^\theta$ is a strictly convex and $C^1$-submanifold of $\fraka_\theta^\ast$. 
\end{lemma}

\begin{proof}
    We first show strict convexity of $\partial \Dc_\Gamma^\theta$. For any $\psi_0,\psi_1\in \partial \Dc_\Gamma^\theta$, we set $\psi_t=(1-t)\psi_0+t\psi_1$ for $t\in [0,1]$. By the convexity of the critical exponent, we have $\delta_{\psi_t}(\Gamma) \leqslant 1$. Then \cite[Corollary~1.10]{canary2025patterson} implies that $\delta_{\psi_t}(\Gamma) < 1$ for any $t\in (0,1)$ when $\psi_0 \ne \psi_1$. This implies that $\partial \Dc_\Gamma^\theta$ is strictly convex.

    Next we show $C^1$-regularity of $\partial \Dc_\Gamma^\theta$. Let $\omega_\al$ be the fundamental weight associated to $\al\in \theta$. Then $\{\omega_\al:\al\in \theta\}$ is a basis of $\fraka_\theta^\ast$, and moreover each $\omega_\al$ is $(\Gamma,\theta)$-proper. Let $\theta=\{\al_0,\al_1,\dots,\al_m\}$ and denote $\p_j=\om_{\al_j}$ for $j=0,\dots,m$. Let $\Theta:\Rb^m\ra \Rb$ be the Manhattan parametrization of $(\msf d_{\p_0},\dots,\msf d_{\p_m})$, and note that 
    \[\partial \Dc_\Gamma^\theta=\left\{\sum_{j=1}^m{v_j\p_j}+\Theta(\vv)\p_0: (v_1,\dots,v_m)\in \Rb^m\right\}~.\]
    Then the $C^1$-regularity of $\partial \Dc_\Gamma^\theta$ follows from \Cref{theorem: manhattan regularity for relatively Anosov representations}.
\end{proof}

\begin{proof}[Proof of \Cref{corollaryalpha: strict concavity of growth indicator}]
    The limit cone $\Lc_\theta(\Gamma)\subset \fraka_\theta$ has non-empty interior since $\Gamma$ is Zariski dense. Let $\psi^\thet_\Gamma:\Lc_\theta(\Gamma)\ra [-\infty,\infty]$ be the $\theta$-growth indicator function. By \cite[Theorem~3.3]{kim2025properly}, $\p_\Gamma^\theta$ satisfies:
    \begin{itemize}
        \item it is finite and nonnegative on $\Lc_\theta(\Gamma)$, and positive on the interior of $\Lc_\theta(\Gamma)$;
        \item it is constant and equal to $-\infty$ outside of $\Lc_\theta(\Gamma)$;
        \item it is homogeneous, upper semicontinuous, and concave.
    \end{itemize}
    Since $\partial\Dc_\Gamma^\theta$ is the boundary of the set $\partial \Dc_\Gamma^\theta=\{\psi\in \fraka_\theta^\ast: \psi\geqslant \p_\Gamma^\theta\}$, we can apply the results from \cite[Section~4.1]{quint2003indicateur} (see also \cite[Lemma~4.8]{sambarino2014hyperconvex}). In particular, $\partial \Dc_\Gamma^\theta$ being $C^1$ implies that $\p_\Gamma^\theta$ is strictly concave on the interior of $\Lc_\theta(\Gamma)$.

    It remains to show that $\psi^\thet_\Gamma$ is $C^1$-regular. For any vector $\vv$ in the interior of $\Lc_\theta(\Gamma)$, we consider the set of supporting linear forms of $\partial \Dc_\Gamma^\theta$ at $\vv$, namely \[\partial \Dc_{\vv} = \{\psi\in\partial \Dc_\Gamma^\theta:\psi(\vv)=\psi^\thet_\Gamma(\vv)\}~.\] Since $\psi^\thet_\Gamma$ is concave, $\partial \Dc_{\vv}$ is non-empty. We claim that $\partial \Dc_{\vv}$ is a singleton. Otherwise there exist two distinct linear forms $\psi_0,\psi_1\in \partial \Dc_\vv$, and hence for every $s\in[0,1]$, the form $\psi_s=(1-s)\psi_0+s\psi_1$ satisfies $\psi_s\geqslant \psi^\thet_\Gamma$ and $\psi_s(\vv)=\psi^\thet_\Gamma(\vv)$. Hence $\partial \Dc_\Gamma^\theta$ contains a line segment, which contradicts the strict convexity of $\partial \Dc_\Gamma^\theta$. Therefore $\partial \Dc_{\vv}$ is a singleton, which confirms the claim. Finally, $\psi^\thet_\Gamma$ is $C^1$-regular by the fact that any concave function on an open convex set is differentiable at a point if and only if its superdifferential at that point is a singleton.
\end{proof}

\bibliographystyle{alpha}
\bibliography{reference.bib}

\end{document}